\documentclass[a4paper,leqno,12pt]{amsart}
\usepackage[all]{xy}
\usepackage{amsmath}
\usepackage{amstext}
\usepackage{amsfonts}
\usepackage[mathscr]{euscript}
\usepackage{amscd}
\usepackage{latexsym}
\usepackage{amssymb}
\usepackage{multirow}
\usepackage{bigstrut}
\usepackage{tikz-cd}
\usepackage{tikz}
\usepackage{mathtools}
\usepackage{graphics}
\usepackage{graphicx}
\usepackage{bm}
\usepackage{relsize}

\usepgflibrary{arrows.meta}
\usepgflibrary{plotmarks}
\pgfsetplotmarksize{0.045cm}
\usetikzlibrary{positioning}

\tikzcdset{arrows={line width=0.6pt},arrow style=tikz}

\tikzset{>=stealth}

\theoremstyle{plain}
\swapnumbers
    \newtheorem{theorem}{Theorem}[section]
    \newtheorem{proposition}[theorem]{Proposition}
    \newtheorem{lemma}[theorem]{Lemma}
    \newtheorem{corollary}[theorem]{Corollary}
    
    \newtheorem{subsec}[theorem]{}
\theoremstyle{definition}
    \newtheorem{definition}[theorem]{Definition}

\theoremstyle{remark}
        \newtheorem{remark}[theorem]{Remark}

    \newtheorem{ack}[theorem]{Acknowledgements}
\newenvironment{myeq}[1][]
{\stepcounter{theorem}\begin{equation}\tag{\thetheorem}{#1}}
{\end{equation}\ignorespacesafterend}
\newcommand{\myrdiag}[2][]
{\stepcounter{theorem}\begin{equation}
     \tag{\thetheorem}{#1}\vcenter{\xymatrix@R=10pt@C=10pt{#2}}\end{equation}}
\newcommand{\mysdiag}[2][]
{\stepcounter{theorem}\begin{equation}
     \tag{\thetheorem}{#1}\vcenter{\xymatrix@R=19pt@C=15pt{#2}}\end{equation}}
\newenvironment{mysubsection}[2][]
{\begin{subsec}\begin{upshape}\begin{bfseries}{#2.}
\end{bfseries}{#1}}
{\end{upshape}\end{subsec}}
\newcommand{\sect}{\setcounter{theorem}{0}\section}

\newcommand{\clC}{\mathcal{C}}
\newcommand{\clD}{\mathcal{D}}

\newcommand{\clF}{\mathcal{F}}
\newcommand{\Fs}{\clF\sb{\bl}}
\newcommand{\Fss}{\ovl{\clF}\sb{\bl}}

\newcommand{\clK}{\mathcal{K}}
\newcommand{\ck}{\mathcal{K}}

\newcommand{\clM}{\mathcal{M}}

\newcommand{\clO}{\mathcal{O}}

\newcommand{\clS}{\mathcal{S}}

\newcommand{\clV}{\mathcal{V}}

\newcommand{\clX}{\mathcal{X}}
\newcommand{\clY}{\mathcal{Y}}
\newcommand{\clZ}{\mathcal{Z}}

\newcommand{\za}{\alpha}
\newcommand{\zb}{\beta}

\newcommand{\zd}{\delta}
\newcommand{\zD}{\Delta}

\newcommand{\zve}{\varepsilon}
\newcommand{\zf}{\phi}
\newcommand{\zl}{\lambda}
\newcommand{\zL}{\Lambda}

\newcommand{\zt}{\theta}
\newcommand{\zvt}{\vartheta}

\newcommand{\zr}{\rho}
\newcommand{\zs}{\sigma}

\newcommand{\pt}{\partial}

\newcommand{\ab}{\operatorname{ab}}
\newcommand{\Aug}{\operatorname{Aug}}
\newcommand{\BW}{\operatorname{BW}}
\newcommand{\Com}{\operatorname{C}}

\newcommand{\Dec}{\operatorname{Dec}}
\newcommand{\dd}{\operatorname{d}\!}
\newcommand{\dX}{\dd\! X}
\newcommand{\dpz}{\dd\pi\sb{0}}

\newcommand{\Diag}{\operatorname{Diag}}
\newcommand{\Ner}{\operatorname{Ner}}

\newcommand{\Hom}{\operatorname{Hom}}
\newcommand{\map}{\operatorname{map}}

\newcommand{\Id}{\operatorname{Id}}

\newcommand{\bj}{\bar{\mbox{\textit{\j}}}}

\newcommand{\opp}{\sp{\operatorname{op}}}

\newcommand{\pr}{\operatorname{pr}}
\newcommand{\SSO}{\operatorname{SO}}
\newcommand{\Tot}{\operatorname{Tot}}
\newcommand{\bsim}{/\!\!\sim}

\newcommand{\bl}{\bullet}
\newcommand{\Cu}{C\sp{\bl}}
\newcommand{\hCu}{\widehat{C}\sp{\bl}}
\newcommand{\nid}{\noindent}

\newcommand{\ovl}[1]{\overline{#1}}

\newcommand{\up}[1]{^{(#1)}}

\newcommand{\rw}{\rightarrow}

\newcommand{\xrw}{\xrightarrow} 
\newcommand{\hxrw}[1]{\xymatrix{\ \ar@{^{(}->}^{#1}[r] & \ }}
\newcommand{\tiund}[1]{{\times}\sb{#1}}
\newcommand{\pro}[3]{#1\tiund{#2}\overset{#3}{\cdots}\tiund{#2}#1}
\newcommand{\tens}[2]{#1\,\tiund{#2}\,#1}
\newcommand{\uset}[2]{\underset{#1}{#2}}
\newcommand{\oset}[2]{\overset{#1}{#2}}
\newcommand{\mi}{\text{-}}

\newcommand{\cop}{\textstyle{\,\coprod\,}}

\newlength{\myline}
\newcommand{\Dop}{\Delta\opp}

\newcommand{\Cat}{\mbox{$\mathsf{Cat}\,$}}

\newcommand{\Cath}{\mbox{$\mathsf{Cat\sb{\clO}}$}}
\newcommand{\Gp}{\mbox{$\mathsf{Gp}$}}
\newcommand{\Gpd}{\mbox{$\mathsf{Gpd}$}}
\newcommand{\Gpdo}{\mbox{$\mathsf{Gpd}\sb{\clO}$}}
\newcommand{\GGpdo}{\mbox{$2\mi\mathsf{Gpd}\sb{\clO}$}}
\newcommand{\Spl}{\mbox{$\mathsf{Spl\,}$}}

\newcommand{\Nb}[1]{N\sb{(#1)}}

\newcommand{\Trt}{\mbox{$\mathsf{Track\sb{\clO}}$}}
\newcommand{\Gro}{\mbox{$\mathsf{Graph\sb{\clO}}$}}
\newcommand{\Set}{\mbox{$\mathsf{Set}$}}
\newcommand{\Top}{\mbox{$\mathsf{Top}$}}

\newcommand{\iov}{\ovl{I}}
\newcommand{\ion}{\ovl{N}}

\newcommand{\funcat}[2]{[\Delta^{{#1}\opp},#2]}

\newcommand{\diov}{\ovl{\Diag}}

\newcommand{\tiq}{\oset{q}{\times}}

\newcommand{\Pz}{\Pi\sb{0}}
\newcommand{\nab}{\nabla}

\newcommand{\wh}{\ -- \ }

\newcommand{\w}[2][ ]{\ensuremath{#2}{#1}}
\newcommand{\ww}[1]{\ensuremath{#1}}
\newcommand{\wb}[2][ ]{(\ensuremath{#2}){#1}}
\newcommand{\wref}[2][ ]{\ \eqref{#2}{#1}\ }

\newcommand{\hs}{\hspace*{5 mm}}
\newcommand{\hsm}{\hspace*{2 mm}}

\newcommand{\Sa}{\clS\sb{\ast}}
\newcommand{\hy}[2]{{#1}\text{-}{#2}}
\newcommand{\hpi}{\hat{\pi}\sb{1}}
\newcommand{\iO}[1]{({#1}\!,\!\clO)}
\newcommand{\iOC}[1]{\hy{\iO{#1}}{\Cat}}
\newcommand{\GO}{(\Gpd,\!\clO)}
\newcommand{\GOC}{\hy{\GO}{\Cat}}
\newcommand{\SO}{\iO{\clS}}
\newcommand{\SaO}{\iO{\Sa}}
\newcommand{\OC}{\hy{\clO}{\Cat}}
\newcommand{\SOC}{\iOC{\clS}}
\newcommand{\SaOC}{\iOC{\Sa}}
\newcommand{\VO}{\iO{\clV}}
\newcommand{\VOC}{\iOC{\clV}}

\newcommand{\BL}{B\Lambda}
\newcommand{\EM}[3]{E\sp{#1}({#2},{#3})}
\newcommand{\EL}[2]{\EM{\Lambda}{#1}{#2}}
\newcommand{\HSO}[3]{H\sp{#1}\sb{\SSO}({#2};{#3})}
\newcommand{\co}[1]{c({#1})}

\newcommand{\Sc}[1]{\scriptstyle{#1}}

\newcommand{\Sz}[1]{\scriptsize{#1}}

\newcommand{\arsx}[1]{\ar@<1.5ex>@/\sb{0}.5pc/@{-}[#1]\ar@<1.5ex>@/\sb{0}.55pc/@{-}[#1]} 
\newcommand{\ardx}[1]{\ar@<-2.5ex>@/^0.5pc/@{-}[#1]\ar@<-2.5ex>@/^0.55pc/@{-}[#1]}

\begin{document}

\title {Comonad Cohomology of Track Categories}

%

%
\author{David Blanc}
\address{Department of Mathematics\\ University of Haifa\\ 3498838 Haifa\\ Israel}
\email{blanc@math.haifa.ac.il}
\author{Simona Paoli}
\address{Department of Mathematics\\ University of Leicester\\ Leicester
  LE1 7RH, UK}
\email{sp424@le.ac.uk}
\date{\today}

\begin{abstract}
  We define a comonad cohomology of track categories, and show that it is related via a long exact sequence
  to the corresponding $\SO$-cohomology. Under mild hypotheses, the comonad cohomology coincides,
  up to reindexing, with the $\SO$-cohomology, yielding an algebraic description of the latter.
  We also specialize to the case where the track category is a $2$-groupoid.
\end{abstract}

\maketitle

\setcounter{section}{0}

%
%
\section*{Introduction}
\label{cint}

One of several models for $(\infty,1)$-categories, a central topic
of study in recent years (see, e.g., \cite{Bergner2010}) is the
category of \emph{simplicial categories}; that is, (small)
categories $Y$ enriched in simplicial sets.  If the object set of
$Y$ is $\clO$, we say it is an \textit{$\SO$-category}.

One may analyze a topological space (or simplicial set) $X$  by means of its \emph{Postnikov tower}
\w[,]{(P\sp{n} X)\sb{n=0}\sp{\infty}} where the $n$-th \emph{Postnikov section} \w{P\sp{n}X} is an
\emph{$n$-type} (that is,  has trivial homotopy groups in dimension greater than $n$). The successive
sections are related through their \emph{$k$-invariants}: cohomology classes in
\w[.]{H\sp{n+1}(P\sp{n-1}X; \pi\sb{n}X)}

Since the Postnikov system is functorial (and preserves products), one can also define it for a
simplicial category $Y$: \w{P\sp{n}Y} is then a category enriched in $n$-types, and its $k$-invariants
are expressed in terms of the \ww{\SO}-cohomology of \cite{DwyerKanSmith1986}.

A long-standing open problem is to find a purely ``algebraic'' description of Postnikov systems, both for
spaces and for simplicial categories. For the Postnikov sections, there are various algebraic models
of $n$-types \wh and thus of categories enriched in $n$-types\wh in the literature, using a variety of
higher categorical structures. However, the problem of finding an algebraic model for the $k$-invariants
is largely open. For this purpose, we need first an algebraic formulation of the cohomology theories
used to define the $k$-invariants. This leads us to look for an algebraic description of the cohomology
of a category enriched in a suitable algebraic model of $n$-types.

We here realize the first step of this program, for \textit{track categories} \wh that is, categories enriched
in groupoids. In the future we hope to extend this to the cohomology of $n$-track categories \wh that is,
those enriched in the $n$-fold groupoidal models of $n$-types developed by the authors in
\cite{BP} and \cite{Paolibook}.

In \cite{BlancPaoli2011} the authors introduced a cohomology theory for  track categories (which
generalizes the Baues-Wirsching cohomology of categories \wh see \cite{BauesWirsching1985}), and showed
that it coincides, up to indexing with the corresponding $\SO$-cohomology.
This was then used to describe the first $k$-invariant for a $2$-track category.

A direct generalization of this approach is problematic, because
of the difficulty of defining a full and faithful simplicial nerve
of weak higher categorical structures. Instead, we use a version
of Andr\'{e}-Quillen cohomology, also known as \textit{comonad
cohomology}, since we use a comonad to produce a simplicial
resolution of our track category (see \cite{BarrBeck}). We
envisage a generalization to higher dimensions, using the $n$-fold
nature of the models of $n$-types in \cite{BP} and
\cite{Paolibook}.

Our main result (see Corollaries \ref{cor-maps-theta-xi-1} and \ref{cor-maps-theta-xi-2}) is that under
mild hypotheses on a track category $X$ (always satisfied up to $2$-equivalence), the comonad
cohomology of $X$ (Definition \ref{def-maps-theta-xi-1}) coincides, up to a dimension shift, with
its $\SO$-cohomology. This follows from Theorem \ref{the-maps-theta-xi-1}, which states that any
track category $X$ has a long exact sequence relating the comonad cohomology of $X$, its
\ww{\SO}-cohomology and the \ww{\SO}-cohomology of the category \w{X\sb{0}} of objects and
$1$-arrows of $X$.
When the track category $X$ is a $2$-groupoid, its \ww{\SO}-cohomology coincides with the cohomology
of its classifying space.

\begin{mysubsection}{Notation and conventions}\label{snac}
Denote by $\Delta$ the category of finite ordered sets, so for any $\clC$, \w{[\Dop,\clC]}
is the category of simplicial objects in $\clC$, while \w{[\Delta,\clC]} is the category
of cosimplicial objects in $\clC$. In particular, we write $\clS$ for the category
\w{[\Dop,\Set]} of simplicial sets. We write \w{\co{A}\in [\Dop,\clC]} for the constant
simplicial object on \w[.]{A\in\clC}

For any category $\clC$  with finite limits, we write \w{\Gpd\,\clC} for the category of
groupoids internal to $\clC$ \wh that is,  diagrams in $\clC$ of the form
\begin{equation*}
  \xymatrix{
    X\sb{1}\times\sb{X\sb{0}}X\sb{1} \ar^(0.6){m}[r] &
    \hspace*{1mm} X\sb{1} \ar@(ul,ur)[]^{c}\ar@<1ex>^(0.5){d\sb{0}}[rr]\ar_(0.5){d\sb{1}}[rr] &&
    X\sb{0} \ar@<2.5ex>^(0.5){\zD}[ll]
    }
\end{equation*}
\nid satisfying the obvious identities making the composition $m$ associative and
every `$1$-cell' in \w{X\sb{1}} invertible.

For a fixed set $\clO$, we denote by \w{\Cath} the category of small categories with
object set $\clO$ (and functors which are the identity on $\clO$).
In particular, a category $\clZ$ enriched in simplicial sets with object
set $\clO$  will be called an \ww{\SO}-\emph{category}, and the category of
all such will be denoted by \w[.]{\SOC} Equivalently, such a category $\clZ$ can
be thought of as a simplicial object in \w[.]{\Cath} This means
$\clC$ has a fixed object set $\clO$  in each dimension, and all face and degeneracy
functors the identity on objects.

More generally, if \w{(\clV,\otimes)} is any monoidal category, a
\ww{\VO}-\emph{category} is a small category \w{\clC\in \Cath} enriched
over $\clV$. The category of all such categories will be denoted by
\w[.]{\VOC} Examples for \w{(\clV,\otimes)} include $\clS$, \w[,]{\Top}
\w[,]{\Gp} and \w[,]{\Gpd}  with $\otimes$ the Cartesian
product. When \w[,]{\clV=\Gpd} we call $\clZ$ in \w{\Trt:=\GOC} a \emph{track category}
with object set $\clO$ (see \S \ref{sbs-track-cat} below).

Another example is pointed simplicial sets \w[,]{\clV=\Sa} with
\w{\otimes=\wedge} (smash product). We can identify an
\ww{\SaO}-category with a simplicial pointed $\clO$-category.
\end{mysubsection}

\begin{mysubsection}{Organization}\label{sbs-org}
  Section \ref{sec-prelim} provides some background material on the Bourne adjunction
  (\S \ref{sbs-bourne}),
internal arrows (\S \ref{sbs-adjpair}), modules ( \S \ref{sbs-modules}), \ww{\SO}-cohomology
(\S\ref{sbs-theta-cohom}) and simplicial model categories (\S \ref{sbs-bous-kan}).
Section \ref{sec-short-seq} sets up a short exact sequence associated to certain internal groupoids
(Proposition \ref{pro-short-exact-1}), and Section \ref{sec-comonad-res} introduces the comonad used
to define our cohomology, and shows its relation to \ww{\SO}-cohomology (Theorem \ref{the-com-res-2}
and Corollary \ref{cor-com-res-1}).
Section \ref{sec-coh-track} defines the comonad cohomology of track categories, and establishes the
long exact sequence relating the \ww{\SO}-cohomologies of $X$ and of \w{X\sb{0}} and the comonad
cohomology of $X$ (Theorem \ref{the-maps-theta-xi-1}).
Section \ref{sec-groupoidal} specializes to the case of a $2$-groupoid, showing that in this case
\ww{\SO}-cohomology coincides with that of the classifying space (Corollary \ref{pro-groupoidal-11}).
The long exact sequence of Corollary \ref{the-groupoidal-1} recovers \cite[Theorem 13]{PaoliHHA2003}.
\end{mysubsection}

\begin{ack}
We would like to thank the referee for his or her pertinent and helpful comments.
The first author was supported by the Israel Science Foundation grants 74/11 and 770/16.
The second author would like to thank the Department of Mathematics of the University
of Haifa for its hospitality during several visits.
\end{ack}

%
%
\sect{Preliminaries}\label{sec-prelim}

In this section we review some background material on the Bourne adjunction, the internal arrow functor,
modules, \ww{\SO}-cohomology, and simplicial model categories.

\begin{mysubsection}{The Bourne adjunction}\label{sbs-bourne}
  Let $\clC$ be a category with finite limits and let \w{\Spl\clC} be the category whose objects are the
  split epimorphisms with a given splitting:
\begin{myeq}\label{eqsplitepic}
\xymatrix{(A \ar@<0.5ex>^{q}[r] & B \ar@<0.5ex>^{t}[l])}
\end{myeq}
\nid Define \w{R:\Gpd\,\clC\to\Spl\clC} by
\begin{equation*}
RX=\xymatrix{(X\sb{1} \ar@<0.5ex>^{d\sb{0}}[r] & X\sb{0} \ar@<0.5ex>^{s\sb{0}}[l])\;.
}
\end{equation*}
Let \w{H:\Spl\clC\to\Gpd\,\clC} associate to
\w{Y=\xymatrix{(A \ar@<0.5ex>^{q}[r] & B \ar@<0.5ex>^{t}[l])}} the object
\begin{equation*}
  HY:\xymatrix{A\tiq A\tiq A \ar^(0.6){m}[r]& A\tiq A
    \ar@<2ex>^(0.6){\pr\sb{0}}[r]\ar_(0.6){\pr\sb{1}}[r] &
    A \ar@<2.5ex>^(0.4){\zD}[l]}
\end{equation*}
\nid of \w[,]{\Gpd\,\clC} where \w{A\tiq A} is the kernel pair of $q$, \w{\zD=(\Id_A,\Id_A)}
is the diagonal map, and
\begin{equation*}
  m=(\pr\sb{0},\pr\sb{2}):A\tiq A\tiq A\cong (A\tiq A)\tiund{A}(A\tiq A)\rw A\tiq A
\end{equation*}
with \w{\pr\sb{i} m=\pr\sb{i}\pi\sb{i}} \wb[,]{i=0,1} where \w[,]{\pi\sb{i}} \w{\pr\sb{i}}
are the two projections.
Note that \w{HY} is an internal equivalence relation in $\clC$ with \w[.]{\Pi\sb{0}(HY)=B}
Consider the following diagram
\begin{myeq}\label{sbs-bourne-eq1}
  \begin{gathered}
\xymatrix@C=30pt@R=18pt{
\Gpd\,\clC \ar^{\Ner}[r]  \ar@<1ex>^{R}[d] & \funcat{}{\clC} \ar@<-1ex>_{\Dec}[d]\\
\Spl \clC \ar_{n}[r] \ar@<1ex>^{H}[u] & \Aug\funcat{}{\clC} \ar@<-1ex>_{+}[u]
}
\end{gathered}
\end{myeq}
\nid where \w{\Aug\funcat{}{\clC}} is the category of augmented simplicial objects in $\clC$, $+$ is
the functor that forgets the augmentation, the \emph{d\'{e}calage} functor \w{\Dec} (
obtained by forgetting the last face operator) is its right adjoint, and $n X$ is the nerve of the internal
equivalence relation associated to $X$, augmented over itself, with \w[.]{\Ner H = +n}

Diagram \wref{sbs-bourne-eq1} commutes up to isomorphism \wh that is, there is a natural
isomorphism \w[.]{\za:\Dec\Ner\cong nR}
Since \w[,]{+ \dashv \Dec} this implies that \w{H \dashv R} (see \cite[Theorem 1]{Bourne1987}).

Given \w{X\in\Spl\clC} as in \wref[,]{eqsplitepic} we have
\w[.]{R H X = (\xymatrix{A\tiq A \ar@<0.5ex>^(0.6){pr\sb{0}}[r] & A \ar@<0.5ex>^(0.4){\zD}[l]})}
The unit \w{\eta:X\rw R H X} of the adjunction \w{H \dashv R} is given by
\begin{myeq}\label{unithr}
\xymatrix@C=40pt@R=18pt{
A \ar@<0.5ex>^{q}[r] \ar_(0.4){t\sb{1}}[d] & B \ar@<0.5ex>^{t}[l] \ar^(0.4){t}[d] \\
A\tiq A \ar@<0.5ex>^(0.6){pr\sb{0}}[r] & A \ar@<0.5ex>^(0.4){\zD}[l]
}
\end{myeq}
where \w{t\sb{1}} is determined by
\begin{equation*}
\xymatrix@C=40pt@R=15pt{
A \ar@/^1pc/^{\Id}[rrd] \ar^{t\sb{1}}[rd] \ar@/_1pc/_{tq}[rdd]& & \\
& A\tiq A \ar^{\pr\sb{1}}[r] \ar_{\pr\sb{0}}[d] & A \ar^{q}[d] \\
& A \ar_{q}[r] & A
}
\end{equation*}
so that
\begin{myeq}\label{sbs-bourne-eq5}
  pr\sb{0} t\sb{1} = tq,\hs\pr\sb{1} t\sb{1} = \Id,\hs\text{and}\hs t\sb{1} t=(t,t)=\zD t\;.
\end{myeq}
This show that \w{\eta=(t,t\sb{1})} is a morphism in \w[.]{\Spl \clC}

Finally, if $\mu$ is the counit of the adjunction \w[,]{H \dashv R}  and \w{\mu'} that of \w[,]{+ \dashv \Dec}
then for any \w{X\in\Gpd\,\clC} the following diagram commutes:
\begin{equation*}
\xymatrix@C=30pt@R=18pt{
\Ner H R X \ar@{=}[d] \ar^{\Ner \mu}[rr] && \Ner X\\
+ n R X \ar @{}[d]|{\wr\parallel} &&\\
+\Dec \Ner X  \ar@{-->}_{\mu'\sb{\Ner X}}[uurr] &&
}
\end{equation*}
Thus \w{\mu=P \Ner \mu} where \w[.]{P \dashv \Ner}
\end{mysubsection}

\begin{mysubsection}{The internal arrow functor}\label{sbs-adjpair}
  Let \w{U:\Gpd\,\clC \rw \clC} be the arrow functor, so \w{UY=Y\sb{1}} for \w[,]{Y\in\Gpd\,\clC} and assume
  $\clC$ is (co)complete with commuting finite coproducts and pullbacks. For any \w{X\in\clC} let
  \w{X_s,X_t} be two copies of $X$, with \w{F:X_s\cop X_t \rw X} the fold map and
 \w{LX\in\Gpd\,\clC} the corresponding internal equivalence relation, so \w{(LX)\sb{0}=X_s\cop X_t}
 and \w[.]{(LX)\sb{1}=(X_s\cop X_t)\oset{F}{\times}(X_s\cop X_t)}
Then \w{\xymatrix{X_s\cop X_t \ar@<0.5ex>^(0.7){F}[r] & X \ar@<0.5ex>^(0.3){i\sb{1}}[l]}}  is an object of
\w{\Spl\clC} and
\begin{myeq}\label{adjpair.eq1A}
  LX=H(\xymatrix{X_s\cop X_t \ar@<0.5ex>^(0.7){F}[r] & X \ar@<0.5ex>^(0.3){i\sb{1}}[l]})
\end{myeq}
\nid (cf.\ \S \ref{sbs-bourne}), where \w{i\sb{1}} is the coproduct structure map. Therefore,
\begin{myeq}\label{adjpair.eq1}
\begin{split}
    &(X_s\cop X_t)\oset{F}{\times} (X_s\cop X_t)=\\
    &(X_s\tiund{X}X_s)\cop (X_s\tiund{X}X_t)\cop (X_t\tiund{X}X_s)\cop (X_t\tiund{X}X_t)
    = X\sb{ss}\cop X\sb{st}\cop X\sb{ts}\cop X\sb{tt}
\end{split}
\end{myeq}
where \w[,]{X\sb{ss}=X\sb{s}\tiund{X}X_s} \w[,]{X\sb{st}=X\sb{s}\tiund{X}X_t} \w[,]{X\sb{ts}=X\sb{t}\tiund{X}X_s} and
\w[.]{X\sb{tt}=X\sb{t}\tiund{X}X_t} Under the identification \eqref{adjpair.eq1} the face and degeneracy maps
of \w{LX} are as follows:

\w{s\sb{0}} includes \w{X_s\cop X_t} into \w[;]{X\sb{ss}\cop X\sb{tt}}
\w{d\sb{0}} sends \w{X\sb{ss}} and \w{X\sb{st}} to \w[,]{X_s} and \w{X\sb{ts}} and \w{X\sb{tt}} to \w[;]{X_t}
and \w{d\sb{1}} sends \w{X\sb{ss}} and \w{X\sb{ts}} to \w[,]{X_s} and \w{X\sb{st}} and \w{X\sb{tt}} to \w[.]{X_t}

To see that  $L$ is left adjoint to $U$, given \w[,]{f:X\rw Y\sb{1}=UY} its adjoint \w{\tilde{f}:LX\rw Y} is given by
\w{\tilde{f}\sb{1}:X\sb{ss}\cop X\sb{st}\cop X\sb{ts}\cop X\sb{tt}\to Y\sb{1}}
(determined by \w[,]{s'\sb{0} d'\sb{0} f: X\sb{ss} \rw Y\sb{1}} \w[,]{f: X\sb{st} \rw Y\sb{1}} \w[,]{f\,\circ\tau: X\sb{ts} \rw Y\sb{1}} and
\w[),]{s\sb{0} d\sb{0} f : X\sb{tt} \rw Y\sb{1}} and \w{\tilde{f}\sb{0}:X_s\cop X_t\to Y\sb{0}} (determined by
\w{d'\sb{0} f:X_s \rw Y\sb{0}} and \w[).]{d'\sb{1} f:X_t\rw Y\sb{0}}  Here \w{\tau: X\sb{ts}\rw X\sb{st}} is the switch map,
with \w[.]{\tau\circ\tau = \Id}

Conversely, given \w{g:LX \rw Y} with \w{g\sb{1}:X\sb{ss}\cop X\sb{st}\cop X\sb{ts}\cop X\sb{tt}\to Y\sb{1}} and
\w[,]{g\sb{0}:X_s\cop X_t\to Y\sb{0}} its adjoint \w[,]{\hat{g}:X\rw Y\sb{1}}  has \w{\hat{g}\sb{0}} determined by
\w{d\sb{0} f:X_s \rw Y\sb{0}} and \w[,]{d\sb{1} f:X_t \rw Y\sb{0}} while \w{\hat{g}\sb{1}} is determined by
\w{s\sb{0} d\sb{0} f: X\sb{ss}\rw Y\sb{1} } and \w{s\sb{0} d\sb{1} f: X\sb{tt}\rw Y\sb{1}} (where \w{f:X\sb{st}\rw Y\sb{1}} is the composite
\w[).]{X\sb{st} \xrw{i} X\sb{ss}\cop X\sb{st}\cop X\sb{ts}\cop X\sb{tt} \xrw{g\sb{1}}Y\sb{1}}
\end{mysubsection}

\begin{mysubsection}{Modules}\label{sbs-modules}
  Recall that an abelian group object in a category $\clD$  with finite products is an object $G$ equipped
  with a unit map \w{\zs:*\rw G} (where $\ast$ is the terminal object),  and \emph{inverse map}
  \w[,]{i:G\rw G} and a \emph{multiplication map} \w{\mu:G\times G\rw G}  which is associative,
  commutative and unital. We require further that
\begin{myeq}\label{sbs-modules-eq1}
\mu\circ (\Id,i)\circ\zD~=~\sigma\circ c\sb{\ast}~,
\end{myeq}
\nid where \w{\zD:G\rw G\times G} is the diagonal, and $c\sb{*}$ the map to $\ast$.

\begin{definition}\label{dxmod}
  Given an object \w{X\sb{0}} in a category $\clC$, we denote by  \w{(\Gpd\,\clC,X\sb{0})} the subcategory
of  \w{\Gpd\,\clC} consisting of those $Y$ with \w{Y\sb{0}=X\sb{0}} (and groupoid maps which are the identity
on \w[).]{X\sb{0}} For \w[,]{X\in(\Gpd\,\clC,X\sb{0})} an \emph{$X$-module}  is an abelian group object
$M$ in the slice category \w[.]{(\Gpd\,\clC, X\sb{0})/X}
Since the terminal object of \w{\clD=(\Gpd\,\clC,X\sb{0})/X} is \w[,]{\Id_X:X\to X}
  and the product of \w{\zr:M\to X} with itself in $\clD$ is \w[,]{\zr p\sb{1}=\zr p_2:M\oset{\zr}{\times}M\to X}
 a unit map for \w{\zr:M\to X} is given by a section \w{\sigma:X\to M} (with \w[),]{\zr\zs=\Id}
 and the multiplication and inverse have the forms
\begin{equation*}
\xymatrix@C=2.5pc@R=1.5pc{
  M\oset{\zr}{\times}M \ar^{\mu}[rr] \ar_{}[rd] && M \ar^{\zr}[ld] && M \ar^{i}[rr] \ar_{\zr}[rd] &&
  M \ar^{\zr}[ld]\\
& X & && & X, &
}
\end{equation*}
\nid respectively. Note that \wref{sbs-modules-eq1} applied to \w{\zr:M\to X} implies that
\begin{myeq}\label{sbs-modules-eq2}
  \mu(\Id,i)\circ\zD_M~=~\zs\zr~,
\end{myeq}
\nid for diagonal \w{\zD_M:M\rw M\oset{\zr}{\times}M} and \w{\zs:X\to M}  the zero map of
\w[.]{\zr:M\to X}
\end{definition}

\begin{remark}\label{rhy}
Suppose that \w[,]{X=HY} for \w{H:\Spl\clC \rw\Gpd\,\clC} as in  \S \ref{sbs-bourne} and
\begin{myeq}\label{splitobj}
  Y=\xymatrix{(X\sb{0} \ar@<0.5ex>^{q}[r] & \pi\sb{0} \ar@<0.5ex>^{t}[l])}\in \Spl\clC~.
\end{myeq}
\nid Thus \w[,]{X\sb{1}=X\sb{0}\tiq X\sb{0}} and an $X$-module \w{M\rw X} is given by
\begin{equation*}
\xymatrix@C=3.5pc@R=1.8pc{
  M\sb{1} \ar^{\zr\sb{1}}[r] \ar@<-0.5ex>_{d\sb{0}}[d] \ar@<0.5ex>^{d\sb{1}}[d] &
  X\sb{0}\tiq X\sb{0} \ar@<-0.5ex>_{\pr\sb{0}}[d] \ar@<0.5ex>^{\pr\sb{1}}[d] \\
X\sb{0} \ar^{\Id}[r] & X\sb{0}
}
\end{equation*}
with \w[.]{\zr\sb{1}=(d\sb{0},d\sb{1})}  Note that the fiber \w{M(a,b)} of \w{\zr\sb{1}}
over each \w{(a,b)\in X\sb{0}\tiq X\sb{0}} is an abelian group, with zero
\w[,]{\zs\sb{1}(a,b)} and the zero map \w{\zf:X\rw M} is given by
\begin{equation*}
  \xymatrix@C=3.5pc@R=2pc{
    X\sb{0}\tiq X\sb{0} \ar^{\zf\sb{1}}[r] \ar@<-1.5ex>_{pr\sb{0}}[d]
    \ar^{pr\sb{1}}[d] & M\sb{1} \ar^{\zr\sb{1}}[r]
    \ar@<-1ex>_{d\sb{0}}[d] \ar^{d\sb{1}}[d] & X\sb{0}\times X\sb{0} \ar@<-1.5ex>_{pr\sb{0}}[d]
    \ar^{\pr\sb{1}}[d] \\
    X\sb{0} \ar@<-1ex>@/_1.1pc/_{\zD}[u] \ar@{=}[r] &
    X\sb{0} \ar@<-1ex>@/_1.1pc/_{s\sb{0}}[u] \ar@{=}[r] &
    X\sb{0} \ar@<-1ex>@/_1.1pc/_{\zD}[u]
}
\end{equation*}
with \w[.]{\zr\sb{1}\zf\sb{1}=\Id} Thus for $Y$ as in \wref[,]{splitobj} the adjoint
\w{\hat{\zf}\in\Hom\sb{\scriptsize{\Spl(\clC)/RH Y}}\left(Y, R(M)\right)}
is the composite
\begin{equation*}
\xymatrix@C=3.5pc@R=1.6pc{
  X\sb{0} \ar^{t\sb{1}}[r] \ar@<-1ex>_{q}[d] &
  X\sb{0}\times X\sb{0} \ar^{\zf\sb{1}}[r] \ar@<-1ex>_{pr\sb{0}}[d] & M\sb{1} \ar@<-1ex>_{d\sb{0}}[d]\\
\pi\sb{0} \ar_{t}[r]\ar@<-1ex>_{t}[u] & X\sb{0} \ar@{=}[r] \ar@<-1ex>_{\zD}[u] & X\sb{0} \ar@<-1ex>_{s\sb{0}}[u]
}
\end{equation*}
where \w{\eta=(t\sb{1},t)} is the unit of the adjunction \w[,]{H \dashv R} as in \wref[.]{unithr}
\end{remark}

\begin{definition}\label{didemp}
The idempotent map in \w{\Spl \clC}
\begin{equation*}
\xymatrix@C=3.5pc@R=1.6pc{
X\sb{0} \ar^{tq}[r] \ar@<-1ex>_{q}[d] &  X\sb{0} \ar@<-1ex>_{q}[d]\\
\pi\sb{0} \ar_{\Id}[r]\ar@<-1ex>_{t}[u] & \pi\sb{0} \ar@<-1ex>_{t}[u]
}
\end{equation*}
\nid induces an idempotent \w{e = H(tq): X=H(Y)\to X} in \w{\Gpd\,\clC}
(for $Y$ as in \wref[).]{splitobj}

Note that \w{e\sb{0}=tq} and \w[.]{e\sb{1}=(e\sb{0},e\sb{0})} We therefore obtain an idempotent operation
$\underline{e}$ on \w[,]{\Hom\sb{\Gpd\,\clC/X}(X,M)} taking \w{f:X\rw M} to \w[.]{fe:X\rw M} We
write \w[.]{fe=\underline{e}(f)}
This sends \w{(a,b)\in X\sb{0}\tiund{q}X\sb{0}} to \w[.]{f(tqa,tqb)} Let \w{e\sp{\ast} M} be the pullback
\begin{myeq}\label{estarm}
\begin{gathered}
\xymatrix@C=3.5pc@R=1.6pc{
e\sp{\ast} M \ar^{r}[r] \ar_{\zr'}[d] & M \ar^{\zr}[d]\\
X \ar_{e}[r] & X
}
\end{gathered}
\end{myeq}
\nid  in \w[.]{\Gpd\,\clC} If we denote the fiber of \w{\zr'} at
\w{(a,b)\in X\sb{1}} by \w[,]{(e^{*}M)\sb{1}(a,b)} we have
\w[,]{(e^{*}M)\sb{1}(a,b)=(a,b)\tiund{(ta,tb)}M\sb{1}(ta,tb)} which is isomorphic under
\w{r\sb{1}(a,b)} to \w[.]{M\sb{1}(ta,tb)} The unit map \w{\zs'} of \w{e^{*}M} is given
\begin{myeq}\label{estarmA}
\begin{gathered}
\xymatrix@C=3.5pc@R=1.5pc{
X \ar@/^1pc/^{\zs e}[rrd] \ar^{\zs'}[rd] \ar@/_1pc/_{\Id}[rdd]& & \\
& e\sp{\ast}M \ar^{r}[r] \ar_{\zr'}[d] & M \ar^{\zr}[d] \\
& X \ar_{e}[r] & X
}
\end{gathered}
\end{myeq}
where
\w{\parbox{20mm}{\xymatrix@R=10pt@C=10pt{X\ar^{\zs}[rr]\ar_{\Id}[rd] && M\ar^{\zr}[ld]\\ & X}}}
is the unit of \w[,]{\zr:M\to X} so in particular
\begin{myeq}\label{sbs-modules-eq3}
  \zs\sb{1}\zD\sb{X\sb{0}} = s\sb{0}:X\sb{0}\rw M\sb{1}\;.
\end{myeq}
Thus for each \w{(a,b)\in X\sb{1}} we have
\w[.]{\zs'(a,b)=((a,b),(\zs e)(a,b))=((a,b),\zs(tqa,tqb))}

The multiplication
\w{(e\sp{\ast} M)\sb{1}\tiund{X\sb{1}}(e\sp{\ast} M)\sb{1}
  \xrw{\mu'\sb{1}}(e\sp{\ast} M)\sb{1}}
on \w{((a,b),m),((a,b),m')} by
\begin{myeq}\label{sbs-modules-eq4}
  \mu'\sb{1}(((a,b),m),((a,b),m'))=((a,b),\mu\sb{1}(m,m'))~,
\end{myeq}
\nid so identifying \w{(e\sp{\ast} M)\sb{1}\tiund{X\sb{1}}(e\sp{\ast} M)\sb{1}} with
\w[,]{X\sb{1}\tiund{X\sb{1}}(M\sb{1}\tiund{X\sb{1}}M\sb{1})}
we have \w[.]{\mu'\sb{1}=(\Id,\mu\sb{1})}

\nid Finally, the zero map of \w{e\sp{\ast} M} is given on \w{((a,b),m)\in(e\sp{\ast} M)\sb{1}} by
\begin{equation*}
  \zs'\sb{1}\zr'\sb{1}((a,b),m)=\zs'\sb{1}(a,b)=\{(a,b),\zs(tqa,tqb)\}=\{(a,b), \zs \zr\sb{1}(m)\}
\end{equation*}
\nid since \w[.]{\zr\sb{1}(m)=(tqa,tqb)} Thus \w[.]{O\sb{(e\sp{\ast} M)\sb{1}}=\zs'\sb{1}\zr'\sb{1}=(\Id,\zs \zr)=(\Id,O\sb{M\sb{1}})}
\end{definition}
\end{mysubsection}

\begin{mysubsection}{{$\pmb{\SO}$}-Categories}\label{sbs-theta-cohom}
In \cite[\S 1]{DwyKan1980-1}, Dwyer and Kan define a simplicial model
category structure on \w[,]{\SOC} also valid for \w{\SaOC} (see \S \ref{snac} and
\cite[Prop.~1.1.8]{HovM}), in which a map \w{f:\clX\rw\clY} is a fibration
(respectively, a weak equivalence) if for each \w[,]{a,b\in\clO} the
induced map \w{f\sb{(a,b)}:\clX(a,b)\rw\clY(a,b)} is such.

The cofibrations in ${\SOC}$ or ${\SaOC}$ are not easy to describe.
However, for any \w[,]{\clK\in\Cath} the constant simplicial category
\w{\co{\clK}\in \funcat{}{\Cath}\cong\SOC} has a cofibrant replacement defined as follows:

Recall that a category \w{Y\in\Cath} is \emph{free} if there exists a set $S$ of
non-identity maps in $Y$ (called generators) such that every non-identity map in $Y$
can uniquely be written as a finite composite of maps in $S$. There is a forgetful functor
\w{U:\Cath\rw\Gro} to the category of directed graphs, with left adjoint the free category
functor \w{F:\Gro\rw\Cath}   (see\cite{Hasse}  and compare \cite[\S 2.1]{DwyKan1980-1}).
Both $U$ and $F$ are the identity on objects.

Similarly, an \ww{\SO}-category \w{X\in\funcat{}{\Cath}=\SO\mi\Cat} is \emph{free} if for each
\w[,]{k\in\Delta} \w{X_k\in\Cath} is free, and the degeneracy maps in $X$ send generators
to generators. Every free \ww{\SO}-category is cofibrant (cf.\ \cite[\S 2.4]{DwyKan1980-1}).
Moreover, for any \w[,]{\clK\in\Cath} a canonical cofibrant replacement \w{\Fs\clK}
for \w{\co{\clK}} in \w{\SO\mi\Cat=\funcat{}{\Cath}} (\S \ref{snac}) is obtained
by iterating the comonad \w{FU:\Cath\rw\Cath} (so \w[).]{\clF\sb{n}\clK:=(FU)^{n+1}\clK}
The augmentation \w{\Fs\clK\rw\clK} induces a weak equivalence
\w{\Fs\clK\simeq\co{\clK}} in \w[.]{\funcat{}{\Cath}\cong\SOC} If $\clK$ is pointed,
\w{\Fs\clK} is a \ww{\SaO}-category.

More generally, if $X$ is any \ww{\SO}-category, thought of as a
simplicial object in \w[,]{\Cath} its \emph{standard Dwyer-Kan
resolution} is the cofibrant replacement given by the diagonal
\w{\Diag\Fss X} of the bisimplicial object \w{\Fss
X\in\funcat{2}{\Cath}} obtained by iterating \w{FU} in each
simplicial dimension.
\end{mysubsection}

\begin{definition}\label{dsocoh}
The \emph{fundamental track category} of an \ww{\SO}-category $Z$ is obtained by
applying the fundamental groupoid functor \w{\hpi:\clS\rw\Gpd} to each mapping space
\w{\clZ(a,b)} (see \cite[\S I.8]{GJardS}. When $Z$ is fibrant, \w{\Lambda:=\hpi Z} has
a particularly simple description: for each \w[,]{a,b\in\clO} the set of objects
of \w{\Lambda(a,b)} is \w[,]{Z(a,b)\sb{0}} and for \w[,]{x,x'\in Z(a,b)\sb{0}} the morphism set
\w{(\Lambda(a.b))(x,x')} is \w[,]{\{\tau\in Z(a,b)\sb{1}~:\ d\sb{0}\tau=x,d\sb{1}\tau=x'\}\bsim}
where $\sim$ is determined by the $2$-simplices of $X$.
Since \w{\hpi} commutes with cartesian products for Kan complexes, it extends to \w{\SOC}
(after fibrant replacement).

A \emph{module} over a track category \w{\Lambda\in\GOC} is an abelian group object $M$ in
\w{\GOC/\Lambda} (see \S \ref{sbs-modules}).
For example, given a (fibrant) \ww{\SO}-category $Z$, for each \w{n\geq 2} we obtain
a \ww{\hpi Z}-module by applying \w{\pi\sb{n}(-)} to each mapping space of $Z$.

For each track category \w[,]{\Lambda\in\GOC} $\Lambda$-module $M$ and \w[,]{n\geq 1} we have a
twisted Eilenberg-Mac~Lane \ww{\SO}-category \w{E=\EL{M}{n}} over $\Lambda$, with
\w{\pi\sb{n}E\cong M} and \w{\pi\sb{i}E\cong 0} for \w{2\leq i\neq n} (see
\cite[\S 1]{DKobsimp} and \cite[\S 1.3(iv)]{DwyerKanSmith1986}).

Given \w[,]{\Lambda\in\GOC} a $\Lambda$-module $M$, and an object
\w{Z\in\SOC}  equipped with a \emph{twisting map}  \w[,]{p:\hpi Z\rw \Lambda}
the $n$-th \ww{\SO}-\emph{cohomology group} of $Z$ \emph{with coefficients in $M$} is
\begin{equation*}
\HSO{n}{Z/\Lambda}{M}:=[Z,\EL{M}{n}]\sb{\Sz\SOC/\BL}=\pi\sb{0}\map\sb{\Sz\SOC/\BL}(Z,\EL{M}{n})~,
\end{equation*}
\nid where \w{\map\sb{\Sz\SOC/A}(Z,Y)} is the sub-simplicial set of
\w{\map\sb{\Sz\SOC}(Z,Y)} consisting of maps over a fixed base $A$
(cf.\ \cite[\S 2]{DwyerKanSmith1986})
Typically, \w[,]{\Lambda=\hpi Z} with $p$ a weak equivalence; if
in addition \w[,]{Z\simeq \BL} we denote \w{\HSO{n}{Z/\Lambda}{M}} simply by
\w[.]{\HSO{n}{\Lambda}{M}}
\end{definition}

\begin{mysubsection}{Simplicial model categories}\label{sbs-bous-kan}
 Recall that a \emph{simplicial} model category $\clM$ is a model category equipped with functors
  \w{X\mapsto X\otimes K} and \w[,]{X\mapsto X^K} natural in \w[,]{K\in\clS} satisfying appropriate axioms
  (cf.\ \cite[Definition 9.1.6]{Hirsch}).

For example,  $\clS$ itself is a simplicial model category, with \w{X\otimes K:=X\times K} and
  \w[,]{X^K:=\map(K,X)} where \w{\map(K,X)\in \clS} has \w[.]{\map(K,X)\sb{n}:=\Hom\sb{\clS}(K\times\zD[n],X)}
    Similarly, \w{\SO\mi\Cat} is also a simplicial model category (see \cite[Proposition 7.2]{DwyKan1980-1}).
\end{mysubsection}

\begin{definition}\label{def-sbs-bous-kan}
  Let $\clM$ be a simplicial model category. The \emph{realization} \w{|X|} of
  \w{X\in\funcat{}{\clM}} is defined to be the coequalizer of the maps
\begin{equation*}
\xymatrix{
  \uset{(\zs:[n]\rw [k])\in\zD}{\cop} X\sb{k}\otimes \zD[n]
  \ar@<0.5ex>^{\zf}[r] \ar@<-0.5ex>_{\psi}[r] &
  \uset{n\geq 0}{\cop} X\sb{n}\otimes \zD[n]~,
}
\end{equation*}
\nid where on the summand indexed by \w[,]{\zs:[n]\rw [k]} $\zf$ is the composite of
\w{\zs\sp{\ast}\otimes 1\sb{\zD[k]}:X\sb{k}\otimes \zD[n]\rw X\sb{n}\otimes\zD[n]}
with the inclusion into the coproduct, and $\psi$ is the composite of
\w{1\sb{X\sb{k}}\otimes \zs\sb{\ast}:X\sb{k}\otimes\zD[n]\rw X\sb{k}\otimes\zD[k]}
with the same inclusion (see \cite[\S VII.3]{GJardS}).

Similarly, if \w{X\in[\zD,\clM]} is a cosimplicial object in $\clM$,
its \emph{total object} \w{\Tot X} is the equalizer of
\begin{equation*}
\xymatrix{
\uset{[n]\in Ob\,\zD}{\prod} (X\sp{n})^{\zD[n]} \ar@<0.5ex>^{\zf}[r] \ar@<-0.5ex>_{\psi}[r] &
\uset{(\zs:[n]\rw [k])\in\zD}{\prod} (X^k)^{\zD[n]}
}
\end{equation*}
\nid with $\phi$ and $\psi$ defined dually (cf.\ \cite[Def. 18.6.3]{Hirsch}).
\end{definition}

The following is a straightforward generalization of \cite[XII 4.3]{BousKan}:

\begin{lemma}\label{lem-sbs-bous-kan}
  If \w{\clM=\funcat{}{\clD}} is a simplicial model category and
  \w[,]{X\in\funcat{}{\clM}\cong\funcat{2}{\clD}} then \w{\Diag X\cong |X|}
and \w[.]{\map\sb{\clM}(\Diag X,K)\cong \Tot \map\sb{\clM}(X,K)}
\end{lemma}

%
%
\sect{Short exact sequences}\label{sec-short-seq}

We now associate to any internal groupoid of the form \w{X=HY} (cf.\ \S \ref{sbs-bourne})  and
$X$-module $M$ a certain short exact sequence of abelian groups (see
Proposition \ref{pro-short-exact-1}).
When \w[,]{\clC=\Cath} this can be rewritten in a more convenient form (see
Proposition \ref{pro-maps-theta-xi-1}); however, in this section we present it in a
more general context, which may be useful in future work.
A similar short exact sequence appears in \cite[Theorem 3.5]{VanOsdol} for $\clC$
an algebraic category (with a different description of the third term).
When \w[,]{\clC=\Gp} it reduces to \cite[Lemma 6]{PaoliHHA2003}, though the
method of proof there is different.

\begin{definition}\label{ddmaps}
Let \w[,]{X=H \xymatrix{(X\sb{0} \ar@<0.5ex>^{q}[r] & \pi\sb{0} \ar@<0.5ex>^{t}[l])}} with
\w{\dX\sb{0}}
the discrete internal groupoid on \w[.]{X\sb{0}} We define \w{j\colon\dX\sb{0} \rw X}
to be the map
\begin{equation*}
\xymatrix@C=4.5pc@R=1.8pc{
  X_0 \ar^(0.4){\zD\sb{X_0}}[r] \ar@<-1.0ex>_{\Id}[d] \ar^{\Id}[d] &
  X_0\tiund{q}X_0 \ar@<-1ex>_{pr_0}[d] \ar^{\pr\sb{1}}[d] \\
X_0 \ar@<-1ex>@/_0.8pc/_{\Id}[u] \ar_{\Id\sb{X_0}}[r] & X_0 \ar@<-1ex>@/_1.1pc/_{\zD\sb{X_0}}[u]
}
\end{equation*}
\nid in \w[.]{(\Gpd\,\clC , X_0)}
Consider the pullback
\begin{myeq}\label{sbs-map-zvt-eq1}
\begin{gathered}
\xymatrix@C=4.5pc@R=1.5pc{
  j\sp{\ast} M\ar^{k}[r] \ar_{\zl}[d] & M\ar^{\zr}[d] \\
\dX\sb{0} \ar_{j}[r] & X
}
\end{gathered}
\end{myeq}
\nid in \w[,]{\Gpd\,\clC/X} where
\w[,]{d\sb{0}=d\sb{1}=\zl\sb{1}:(j\sp{\ast} M)\sb{1} \rw X\sb{0}} since \w{\dX\sb{0}}
is discrete. Because \wref{sbs-map-zvt-eq1} induces
\begin{myeq}\label{eqjone}
\xymatrix@C=3.5pc@R=2.0pc{
  (j\sp{\ast} M)\sb{1} \ar^{k\sb{1}}[r] \ar_{\zl\sb{1}}[d] & M\sb{1}\ar^{\zr\sb{1}}[d] \\
  X\sb{0} \ar_(0.4){j\sb{1}=\zD\sb{X\sb{0}}}[r] & X\sb{0}\tiund{q}X\sb{0}=X\sb{1}
}
\end{myeq}
\nid (a pullback in $\clC$), we shall denote \w{(j\sp{\ast} M)\sb{1}} by \w[.]{j\sb{1}\sp{\ast} M\sb{1}}

We shall use the following abbreviations for the relevant \w{\Hom} groups:
\begin{myeq}\label{eqabbrev}
 \begin{cases}
   \Hom(\pi\sb{0},\,t\sp{\ast} j\sb{1}\sp{\ast}M\sb{1})~:=&~
   \Hom\sb{\clC/\pi\sb{0}}\,((\pi\sb{0}\xrw{\Id}\pi\sb{0} ),\ (t\sp{\ast}j\sb{1}\sp{\ast}M\sb{1}\xrw{r\sb{1}}\pi\sb{0}))~,\\
   \Hom(X\sb{0},\,j\sb{1}\sp{\ast} M\sb{1})~:=&~
   \Hom\sb{\clC/X\sb{0}}\,((X\sb{0}\xrw{tq}X\sb{0})),\ (j\sb{1}\sp{\ast}M\sb{1}\xrw{\zl\sb{1}}X\sb{0}))~,\\
  \Hom(X\sb{1},e\sp{\ast}M\sb{1})~:=&~
  \Hom\sb{\clC/X\sb{1}}\,((X\sb{1}\xrw{\Id}X\sb{1})),\ (((e\sp{\ast}M)\sb{1}\xrw{\zr\sb{1}}X\sb{1})))~,\\
  \Hom(X,e\sp{\ast}M)~:=&~
  \Hom\sb{\Sz\Gpd\, \clC/X}\,((X\!\xrw{\Id}\!X),\ (e\sp{\ast}M\xrw{\zr'}X))~.
 \end{cases}
\end{myeq}
\end{definition}

 \begin{definition}\label{dtheta}
   In the situation described in \S \ref{ddmaps}, given a map
   \w{\parbox{20mm}{\xymatrix@C=1.5pc@R=1.0pc{
X\sb{0} \ar^{f}[rr] \ar_{e\sb{0}=tq}[rd] && j\sb{1}\sp{\ast} M\sb{1} \ar^{\zl\sb{1}}[ld]\\ & X\sb{0} &}}}
   in \w[,]{\clC/X\sb{0}} we have \w[,]{  \zr\sb{1} k\sb{1} f=\zD\sb{X\sb{0}}\zl\sb{1}f =\zD\sb{X\sb{0}}e\sb{0}=e\sb{1} \zD\sb{X\sb{0}}}
since the square
\begin{myeq}\label{sbs-map-zvt-eq2}
\begin{gathered}
\xymatrix@C=3.5pc@R=1.5pc{
X\sb{0} \ar^{e\sb{0}}[r] \ar_{\zD\sb{X\sb{0}}}[d] & X\sb{0} \ar^{\zD\sb{X\sb{0}}}[d]\\
X\sb{1} \ar_{e\sb{1}}[r] & X\sb{1}
}
\end{gathered}
\end{myeq}
\nid commutes. Now let \w{v:X\sb{0} \rw (e\sp{\ast} M)\sb{1}} be given by
\begin{myeq}\label{sbs-map-zvt-eq3}
\begin{gathered}
\xymatrix{
  X\sb{0} \ar@/^1pc/^{k\sb{1} f}[rrd] \ar^{v}[rd]
  \ar@/_1pc/_{\zD\sb{X\sb{0}}e\sb{0}=e\sb{1}\zD\sb{X\sb{0}}}[rdd]& & \\
& (e\sp{\ast}M)\sb{1} \ar^{r}[r] \ar_{\zr'\sb{1}}[d] & M\sb{1} \ar^{\zr\sb{1}}[d] \\
& X\sb{1} \ar_{e\sb{1}}[r] & X\sb{1}
}
\end{gathered}
\end{myeq}
where
\w{\zr\sb{1} k\sb{1} f=\zD\sb{X\sb{0}}\zl\sb{1} f=\zD\sb{X\sb{0}} e\sb{0}=
  e\sb{1}\zD\sb{X\sb{0}}=e\sb{1}e\sb{1}\zD\sb{X\sb{0}}}
by \wref{sbs-map-zvt-eq1} and \wref[.]{sbs-map-zvt-eq2}

Since \w[,]{e\sb{0}=tq} the following diagram commutes:
\begin{myeq}\label{sbs-map-zvt-eq4}
\begin{gathered}
\xymatrix@C=3.5pc@R=1.5pc{
X\sb{0} \ar^{q}[r] \ar_{v}[dd] & \pi\sb{0} \ar^{t}[d]\\
& X\sb{0} \ar^{\zD\sb{X\sb{0}}}[d]\\
(e\sp{\ast} M)\sb{1} \ar_{\zr'\sb{1}}[r] & X\sb{1}
}
\end{gathered}
\end{myeq}
so $v$ induces \w[.]{(v,v): X\sb{1} = X\sb{0}\tiund{q}X\sb{0} \rw (e\sp{\ast} M)\sb{1}\tiund{X\sb{1}}(e\sp{\ast} M)\sb{1}}
In the notation of \wref[,]{eqabbrev} we may therefore define
\w{\zvt: \Hom(X\sb{0},\,j\sb{1}\sp{\ast} M\sb{1})\to\Hom(X\sb{1},e\sp{\ast}M\sb{1})}
  by letting \w{\zvt\sb{1}(f):X\sb{1}\rw (e\sp{\ast} M)\sb{1}} be the composite
\begin{myeq}\label{sbs-map-zvt-eq5}
  X\sb{1}\xrw{(v,v)}(e\sp{\ast} M)\sb{1}\tiund{X\sb{1}}(e\sp{\ast} M)\sb{1}
  \xrw{(\Id,i\sb{1})}(e\sp{\ast} M)\sb{1}\tiund{X\sb{1}}(e\sp{\ast} M)\sb{1} \xrw{\mu\sb{1}} (e\sp{\ast} M)\sb{1}
\end{myeq}
\nid where \w{i:e\sp{\ast} M\rw e\sp{\ast} M} is the inverse map for the abelian group structure on
\w[.]{e\sp{\ast}M}
\end{definition}

\begin{lemma}\label{lem-sbs-map-zvt-1}
  The map \w{\zvt=(e\sb{0},\zvt\sb{1})} lands in \w[,]{\Hom(X,e\sp{\ast}M)} for
  \w{e\sb{0}=tq} and \w{\zvt\sb{1}} as in \wref[.]{sbs-map-zvt-eq5}
\end{lemma}

\begin{proof}
By \wref{sbs-modules-eq2} we have \w[,]{\zs'\sb{1} \zr'\sb{1} = \mu'\sb{1}(\Id,i'\sb{1})\zD\sb{(e\sp{\ast} M)\sb{1}}} so
\begin{myeq}\label{sbs-map-zvt-eq7}
\begin{gathered}
\xymatrix{
  \Sc{X\sb{0}} \ar^(0.4){v}[r] \ar_{e\sb{1}\zD\sb{X\sb{0}}}[rd] & \Sc{(e\sp{\ast} M)\sb{1}} \ar^(0.3){\zD\sb{(e\sp{\ast} M)\sb{1}}}[r]\ar^{\zr'\sb{1}}[d] &
  \Sc{(e\sp{\ast} M)\sb{1}\tiund{X\sb{1}}(e\sp{\ast} M)\sb{1}} \ar^{(\Id,i\sb{1})}[r] &
  \Sc{(e\sp{\ast} M)\sb{1}\tiund{X\sb{1}}(e\sp{\ast} M)\sb{1}} \ar^(0.6){\mu\sb{1}}[r] & \Sc{(e\sp{\ast} M)\sb{1}}\\
& \Sc{X\sb{1}} \ar_{\zs'\sb{1}}[rrru]
}
\end{gathered}
\end{myeq}
\nid commutes, as does
\begin{myeq}\label{sbs-map-zvt-eq8}
\begin{gathered}
  \xymatrix@C=4.5pc@R=1.7pc{
X\sb{0} \ar^{\zD\sb{X\sb{0}}}[r] \ar_{v}[d] & X\sb{1} \ar^{(v,v)}[d]\\
(e\sp{\ast} M)\sb{1} \ar_(0.4){\zD\sb{(e\sp{\ast} M)\sb{1}}}[r] & (e\sp{\ast} M)\sb{1}\tiund{X\sb{1}}(e\sp{\ast} M)\sb{1}
}
\end{gathered}
\end{myeq}
\nid so by \wref[,]{sbs-map-zvt-eq7} \wref[,]{sbs-map-zvt-eq8} \wref[,]{sbs-modules-eq3} and
the definition of \w{\zvt\sb{1}(f)} we see that
\begin{myeq}\label{sbs-map-zvt-eq9}
  \zvt\sb{1}(f)\zD\sb{X\sb{0}}=  \mu\sb{1}(\Id,i\sb{1})(v,v)\zD\sb{X\sb{0}}=\zs'\sb{1} e\sb{1}\zD\sb{X\sb{0}}=\zs'\sb{1}\zD\sb{X\sb{0}}e\sb{0}=s'\sb{0} e\sb{0}
\end{myeq}

By \wref[,]{sbs-modules-eq4} for each \w{(a,b)\in X\sb{0}\tiund{q}X\sb{0}} (with \w[)]{tqa=tqb} we have
\begin{myeq}\label{sbs-map-zvt-eq10prev}
  \zvt\sb{1}(f)(a,b)= ((tqa,tqa),\mu\sb{1}(k\sb{1} f a, i\sb{1} k\sb{1} f b))~,
\end{myeq}
so \w{d'\sb{0}\zvt\sb{1}(f)(a,b)=tqa=e\sb{0}\pr\sb{0}(a,b)} and \w[.]{d'\sb{1}\zvt\sb{1}(f)(a,b)=tqb=e\sb{0}\pr\sb{1}(a,b)}
Thus
\begin{myeq}\label{sbs-map-zvt-eq10}
  d_i\zvt\sb{1}(f)=e\sb{0}\pr_i
\end{myeq}
\nid for \w[.]{i=0,1}
Thus \eqref{sbs-map-zvt-eq9} and \eqref{sbs-map-zvt-eq10} show that
\begin{myeq}\label{sbs-map-zvt-eq6}
\begin{gathered}
\xymatrix@C=4.5pc@R=2pc{
  X\sb{1} \ar^(0.4){\zvt\sb{1}(f)}[r] \ar@<-1.0ex>_{\pr_0}[d] \ar^{\pr\sb{1}}[d] &
  (e\sp{\ast}M)\sb{1} \ar@<-1ex>_{d'_0}[d] \ar^{d'\sb{1}}[d] \\
X_0 \ar@<-1ex>@/_1.1pc/_{\zD\sb{X_0}}[u] \ar_{e_0}[r] & X_0 \ar@<-1ex>@/_0.8pc/_{s'_0}[u]
}
\end{gathered}
\end{myeq}
\nid commutes.
Now, given \w{(a,b,c)\in X\sb{0}\tiund{q}X\sb{0}\tiund{q}X\sb{0}} (with \w[),]{tqa=tqb=tqc} we have
$$
c'(\zvt\sb{1}(f)(a,b)\times \zvt\sb{1}(f)(b,c))= ((tqa,tqa),\mu\sb{1}(k\sb{1} f a,i\sb{1} k\sb{1} fb)\circ \mu\sb{1}(k\sb{1} f b, i\sb{1} k\sb{1} f c))
$$
\nid while \w[,]{\zvt\sb{1}(f)c''(a,b,c)=\zvt\sb{1}(f)(a,c)=((tqa,tqa),\mu\sb{1}(k\sb{1} fa,i\sb{1} k\sb{1} f c))} where we denoted by \w{c'':X_1 \tiund{X_0}X_1 \rw X_1} and \w{c':(e^*M)_1 \tiund{X_0} (e^*M)_1} the compositions.

Since \w{\mu\sb{1}} is a map of groupoids, by the interchange rule we see that
$$
\mu\sb{1}(k\sb{1} f a,i\sb{1} k\sb{1} f b)\circ \mu\sb{1}(k\sb{1} f b,i\sb{1} k\sb{1} f c)=\mu\sb{1}(k\sb{1} f a,i\sb{1} k\sb{1} f c)
$$
\nid where $\circ$ the groupoid composition.

Finally, \w[,]{\zr'\sb{1}\circ\zvt\sb{1}(f)=e\sb{1}} so \w{(e\sb{0},\zvt\sb{1}(f))} is indeed
a morphism in \w[.]{\Gpd\,\clC/X}
\end{proof}

\begin{definition}\label{dxi}
The pullback
\begin{myeq}\label{tjstarm}
\begin{gathered}
\xymatrix@C=3pc@R=1.7pc{
t\sp{\ast}j\sp{\ast} M \ar^{l}[r] \ar_{r}[d] & j\sp{\ast} M\ar^{\zl}[d]\\
d \pi\sb{0}  \ar_{\dd t}[r] & \dX\sb{0}
}
\end{gathered}
\end{myeq}
\nid in \w{\Gpd\,\clC/d\pi\sb{0}} gives rise to a pullback
\begin{myeq}\label{sbs-map-zvt-eq12}
\begin{gathered}
\xymatrix@C=3pc@R=1.6pc{
t\sp{\ast}j\sb{1}\sp{\ast} M\sb{1} \ar^{l\sb{1}}[r] \ar_{r\sb{1}}[d] & j\sb{1}\sp{\ast} M\sb{1}\ar^{\zl\sb{1}}[d]\\
\pi\sb{0}  \ar_{t}[r] &  X\sb{0}
}
\end{gathered}
\end{myeq}
\nid in $\clC$, so we may define
\begin{myeq}\label{def-of-xi}
  \xi:\Hom(\pi\sb{0},\,t\sp{\ast} j\sb{1}\sp{\ast}M\sb{1})\to \Hom(X\sb{0},\,j\sb{1}\sp{\ast} M\sb{1})
\end{myeq}
by sending
\w[,]{\parbox{20mm}{\xymatrix@C=1.5pc@R=1.0pc{\pi\sb{0}  \ar^{f}[rr] \ar_{\Id}[rd] & &
      t\sp{\ast} j\sb{1}\sp{\ast} M\sb{1} \ar^{r\sb{1}}[ld]\\ & \pi\sb{0}  &}}} to the map \w{\xi(f)} given by
\begin{myeq}\label{sbs-map-zvt-eq13}
\begin{gathered}
\xymatrix@C=2pc@R=1.6pc{
X\sb{0} \ar^{l\sb{1} f q}[rr] \ar_{tq}[rd] & &  j\sb{1}\sp{\ast} M\sb{1} \ar^{\zl\sb{1}}[ld]\\
& X\sb{0}  &
}
\end{gathered}
\end{myeq}
\end{definition}

\begin{proposition}\label{pro-short-exact-1}
  Given \w{X=H(Y)\in \Gpd\,\clC} as in \wref{splitobj} and
  \w[,]{M\in [(\Gpd\,\clC,X\sb{0})/X]\sb{\ab}} there is a short exact sequence of abelian groups
\begin{equation*}
  0~\to~\Hom(\pi\sb{0},\,t\sp{\ast} j\sb{1}\sp{\ast}M\sb{1})~\xrw{\xi}~
  \Hom(X\sb{0},j\sp{\ast}M\sb{1})~\xrw{\zvt}~\Hom(X,e\sp{\ast}M)~\to~0~,
\end{equation*}
\nid in the notation of \wref[,]{eqabbrev} where $\xi$ is as in \wref{def-of-xi} and $\zvt$ is as in
Lemma \ref{lem-sbs-map-zvt-1}.
\end{proposition}

\begin{proof}
  We first show that \w[:]{\mathrm{Im}\,\xi\subseteq \mathrm{\ker}\zvt} Given
  \w{f':\pi\sb{0} \to t\sp{\ast} j\sb{1}\sp{\ast} M\sb{1}} in \w[,]{\Hom(\pi\sb{0},\,t\sp{\ast} j\sb{1}\sp{\ast}M\sb{1})}
  the map \w{\xi(f')\in \Hom(X\sb{0},j\sp{\ast}M\sb{1})}
  is given by the composite \w{X\sb{0} \xrw{q} \pi\sb{0}  \xrw{f'} t\sp{\ast} j\sb{1}\sp{\ast} M\sb{1} \xrw{l\sb{1}} j\sb{1}\sp{\ast} M\sb{1}}

By \wref{sbs-map-zvt-eq10prev} and \wref[,]{sbs-map-zvt-eq13} for each \w{(a,b)\in X\sb{1}} we have
\begin{equation*}
  \zvt\sb{1}(\xi(f'))(a,b)=\{(tqa,tqb),\mu\sb{1}(k\sb{1} l\sb{1} f' q(a),i\sb{1} k\sb{1} l\sb{1} f' q(b))\}\;.
\end{equation*}
\nid Since \w[,]{q(a)=q(b)} we have \w[,]{\zvt\sb{1}(\xi(f'))(a,b)=((tqa,tqb),0)}
which is the zero map of \w[.]{\Hom\sb{\Sz\Gpd\,\clC/X}(X, e\sp{\ast} M)} This shows that
\w[.]{\mathrm{Im}\;\xi\subseteq \ker \zvt}

Given \w{g':  X\sb{0} \to j\sb{1}\sp{\ast} M\sb{1} } (as in \wref[)]{sbs-map-zvt-eq13} in \w[,]{\ker\zvt\sb{1}} for all
\w{(a,b)\in X\sb{1}} we have
\begin{equation*}
  \zvt(g')(a,b)= \{(tqa,tqa),\mu\sb{1}(k\sb{1} g'(a), i\sb{1} k\sb{1} g'(b)\}= \{(tqa,tqb),0\}\;.
\end{equation*}
\nid Thus \w[,]{k\sb{1} g' pr\sb{0}(a,b)=k\sb{1} g' a =k\sb{1} g' b = k\sb{1} g' pr\sb{1}(a,b)} so that
\begin{myeq}\label{short-exact-eq1}
  k\sb{1} g' pr\sb{0} = k\sb{1} g' pr\sb{1}\;.
\end{myeq}
\nid Since \w{\xymatrix{
    X\sb{1} \ar@<0.5ex>^{pr\sb{0}}[r]\ar@<-0.5ex>_{pr\sb{1}}[r] & X\sb{0} \ar^{q}[r] & \pi\sb{0}}} is a coequalizer, it follows
from \eqref{short-exact-eq1} that there is a map \w{f:\pi\sb{0} \rw M\sb{1}} with \w[,]{f q=k\sb{1} g'}
and thus \w[,]{f=fqt=k\sb{1} g' t} so
\w[.]{\zr\sb{1} f=\zr\sb{1} k\sb{1} g' t=\zD\sb{X\sb{0}}\zl\sb{1} g' t = \zD\sb{X\sb{0}}t q t=\zD\sb{X\sb{0}}t}
Hence there is \w{\ovl{f}:\pi\sb{0} \rw j\sb{1}\sp{\ast} M\sb{1}} defined by \w{f\top t:\pi\sb{0}\to M\sb{1}\times X\sb{0}} into the
pullback \wref[.]{eqjone} Since \w[,]{\zl\sb{1} \ovl{f}=t} there is also a map \w{f':\pi\sb{0} \rw t\sp{\ast} j\sb{1}\sp{\ast} M\sb{1}}
defined by \w{\ovl{f}\top \Id:\pi\sb{0}\to j^{\ast}\sb{1}M\sb{1}\times\pi\sb{0}} into the
pullback \wref[.]{sbs-map-zvt-eq12}

By \wref{short-exact-eq1} and the above we have \w[.]{k\sb{1} g' = fq = k\sb{1} \ovl{f} q}
Since \w{k\sb{1}} is monic, this implies that \w[,]{g'=\ovl{f}q}
and since \w[,]{\ovl{f}= l\sb{1} f'} also \w[.]{g' = l\sb{1} f' q = \zvt\sb{1}(f')}
This shows that \w[.]{\ker \zvt \subseteq \mathrm{Im}\,\xi} In conclusion, \w[.]{\ker \zvt= \mathrm{Im}\,\xi}

To show that $\xi$ is monic, assume given
\w{f, g \in \Hom(\pi\sb{0},\,t\sp{\ast} j\sb{1}\sp{\ast}M\sb{1})} with \w[.]{\xi f=\xi g}
Then \w[,]{l\sb{1} f q = l\sb{1} g q} which implies that
\begin{myeq}\label{short-exact-eq3}
  l\sb{1} f~=~l\sb{1} g~,
\end{myeq}
\nid since $q$ is epic. Also, \w{\zr\sb{1} k\sb{1} l\sb{1} f = \zD\sb{X\sb{0}} t s\sb{1} f} and \w[,]{\zr\sb{1} k\sb{1} l\sb{1} g = \zD\sb{X\sb{0}} t s\sb{1} g}
so by \wref{short-exact-eq3} we have  \w{\zD\sb{X\sb{0}}t s\sb{1} f = \zD\sb{X\sb{0}} t s\sb{1} g} and therefore
\w[.]{s\sb{1} f = q pr\sb{0} \zD\sb{X\sb{0}} t s\sb{1} f = q pr\sb{0} \zD\sb{X\sb{0}} t s\sb{1} g = s\sb{1} g}
By the definition of \w{t\sp{\ast} j\sb{1}\sp{\ast} M\sb{1}} as the pullback \wref[,]{sbs-map-zvt-eq12}
together with \wref{short-exact-eq3} this implies that \w[.]{f=g} Thus $\xi$ is a monomorphism.

To show that $\zvt$ is onto, assume given  \w{\zf\in\Hom\sb{\Gpd\,\clC/X}(X, e\sp{\ast} M)} (so that
\w[).]{\zr'\circ\phi=e} The adjoint of $\zf$ in
\w[,]{\Hom\sb{\Sz{\Spl \clC} /Y} (Y,\,R(e\sp{\ast} M))} for $Y$ as in \wref[,]{splitobj} is given by a commuting
triangle in \w[:]{\Spl\clC}
\begin{myeq}\label{short-exact-eq6}
\begin{gathered}
\xymatrix@C=5pc@R=1.7pc{
  X\sb{0} \ar^{g}[rr] \ar@<-0.5ex>_{q}[d] \ar_(0.7){(t,tq)}[rdd] & &
  (e\sp{\ast} M)\sb{1} \ar@<-0.5ex>_{d\sb{0}}[d] \ar^(0.7){\zr'\sb{1}}[ldd] \\
\pi\sb{0} \ar@<-0.5ex>_{t}[u] \ar_{t}[rr] \ar_{t}[rdd] & & X\sb{0}  \ar@<-0.5ex>_{s\sb{0}}[u] \ar^{\Id}[ldd]\\
& X\sb{1} \ar@<-0.5ex>_(0.4){pr\sb{0}}[d] &\\
& X\sb{0}  \ar@<-0.5ex>_(0.6){\zD\sb{X\sb{0}}}[u] &
}
\end{gathered}
\end{myeq}
\nid The adjoint $\phi$ of \w{(g,t)} is given by postcomposing with the counit
of \w[,]{H \dashv R} so $\phi$ is the horizontal composite in:
\begin{myeq}\label{short-exact-eq7}
\begin{gathered}
\xymatrix@C=5pc@R=1.8pc{
X_0\tiund{q}X_0 \ar^(0.4){(g,g)}[r] \ar@<-1.0ex>_*-<0.55em>{\sb{pr_0}}[d] \ar^*-<0.50em>{\sb{pr\sb{1}}}[d] \ar_(0.7){g}[rdd] & (e^* M)\sb{1}\tiund{d'_0}(e^*M)\sb{1} \ar^{\mu\sb{1}}[r] \ar@<-1.0ex>_{pr_0}[d] \ar^{pr\sb{1}}[d] & (e^* M)\sb{1} \ar@<-1ex>_{d'_0}[d] \ar^{d'\sb{1}}[d] \ar^(0.7){\zr'\sb{1}}[ldd] \\
X_0 \ar@<2ex>@/^0.8pc/^{\zD\sb{X_0}}[u] \ar_(0.4){g}[r] \ar_{e_0}[rdd] & (e^* M)\sb{1} \ar_(0.6){d'\sb{1}}[r] \ar@<-1.5ex>@/_0.8pc/_{\zD\sb{(e^* M)\sb{1}}}[u] & X_0 \ar@<-1ex>@/_0.8pc/_{\,s'_0}[u] \ar^{\Id}[ldd]\\
& X_0\tiund{q}X_0 \ar@<-1.0ex>_*-<0.50em>{\sb{pr_0}}[d] \ar^*-<0.50em>{\sb{pr\sb{1}}}[d]
&\\
& X_0 \ar@<-1ex>@/_0.8pc/_(0.65){\zD\sb{X_0}}[u]&
}
\end{gathered}
\end{myeq}
\nid By \S \ref{sbs-bourne}, the counit $\mu$ is \w[.]{\mu\sb{1}\{((a,b),m),((a,b),m')\}= ((a,b),m\circ m')}
Since \w{t\sb{1} tq=\zD tq} by \wref[,]{sbs-bourne-eq5}) by \wref{short-exact-eq6} we have
\w[.]{\zr'\sb{1} g= t\sb{1} tq = \zD tq = \zD e\sb{0}}

We have a map \w{g:X\sb{0}\to(e^{*}M)\sb{1}} into the pullback square of \wref{sbs-map-zvt-eq3} given by
\w[.]{\sigma\sb{1} j\sb{1}\top t\sb{1} tq:X\sb{0}\to M\sb{1}\times X\sb{1}}
This in turn defines \w[,]{g':X\sb{0} \rw j\sp{\ast}(e\sp{\ast} M)\sb{1}} given by another pullback:
\begin{equation*}
\xymatrix@C=3pc@R=1.7pc{
X\sb{0} \ar@/^1pc/^{g}[rrd] \ar^{g'}[rd] \ar@/_1pc/_{tq}[rdd]& & \\
& j\sp{\ast}(e^{*}M)\sb{1} \ar^{k\sb{1}}[r] \ar_{}[d] & (e\sp{\ast} M)\sb{1} \ar^{\zr'\sb{1}}[d] \\
& X\sb{0} \ar_{\zD\sb{X\sb{0}}}[r] & X\sb{1}
}
\end{equation*}

We define \w{g'\sb{M\sb{1}}:X\sb{0}\to j^{\ast}M\sb{1}} into the pullback \wref{eqjone} by
  \w[,]{\sigma\sb{1} j\sb{1}\top tq:X\sb{0}\to M\sb{1}\times X\sb{0}} since \w[.]{\zD\sb{X\sb{0}}tq=\zr\sb{1}\sigma\sb{1} j\sb{1}}
By the definitions of $g$, \w[,]{g'\sb{M\sb{1}}} and $\zvt$, for each
\w{(a,b)\in X\sb{1}} we have \w[,]{\zvt(g'\sb{M\sb{1}})(a,b)=\{(tqa,tqb),\mu\sb{1}(g'\sb{M\sb{1}}(a), i\sb{1} k\sb{1} g'\sb{M\sb{1}}(b))\}}
while by \wref{short-exact-eq7} we have:
\begin{equation*}
\begin{split}
  \zf(a,b)= & \{(tqa, tqb),\eta\sb{1}((tqa,tqb), k\sb{1} g'\sb{M\sb{1}}(a), (tqa,tqb), k'\sb{1} g'\sb{M\sb{1}}(b)\}= \\
    = & ((tqa,tqb), g'\sb{M\sb{1}}(a)\circ i\sb{1} k\sb{1} g'\sb{M\sb{1}}(b))~,
\end{split}
\end{equation*}
with \w[.]{g'\sb{M\sb{1}}(a)\circ i\sb{1} k\sb{1} g'\sb{M\sb{1}}(b)\in M\sb{1}(tqa,tqb)} By the Eckmann-Hilton argument,
the abelian group structure on \w{M\sb{1}(tqa,tqb)} is the same as the groupoid structure, so
\w[.]{g'\sb{M\sb{1}}(a)\circ i\sb{1} k\sb{1} g'\sb{M\sb{1}}(b)=\mu\sb{1}(k\sb{1}g'\sb{M\sb{1}}(a), i\sb{1} k\sb{1} g'\sb{M\sb{1}}(b))}
We conclude that \w{\zvt(g')(a,b) = \zf(a,b)} for each \w[,]{(a,b)\in X\sb{1}} so \w[,]{\zvt(g')=\zf} as required.
\end{proof}

We now rewrite the map $\zvt$ of Proposition \ref{pro-short-exact-1} in a different form
(see Lemma \ref{lem-maps-theta-xi-3}). This will be used in Proposition \ref{pro-maps-theta-xi-1}
in the case \w[:]{\clC=\Cath}

\begin{definition}\label{dphi}
Given \w{X=H(Y)\in \Gpd\,\clC} as in \wref[,]{splitobj} we define a map
\begin{myeq}\label{def-thetaprime}
  \zvt':\Hom\sb{\clC/X\sb{0}}(X\sb{0},\,j\sb{1}\sp{\ast} M\sb{1})\rw \Hom\sb{\Sz\Gpd\, \clC/X}(X,M)
\end{myeq}
\nid as follows: the pullback square \wref{eqjone} implies that a  map \w{X\sb{0}\to j\sb{1}\sp{\ast} M\sb{1}} is
given by \w{f:X\sb{0}\to M\sb{1}}  with \w[.]{\zr\sb{1} f=\zD\sb{X\sb{0}}}
We then define \w{\bar\zvt(f):X\sb{1}\rw M\sb{1}} by \w[,]{\bar\zvt(f)(a,b)=\zs(a,b)\circ f(b)-f(a)\circ \zs(a,b)}
where $\circ$ is the groupoid composition in $M$ and the subtraction is that of the abelian group
object $M$. This gives rise to a map
\begin{equation*}
\xymatrix@C=4pc@R=2pc{
X\sb{1} \ar^{\bar\zvt(f)}[r] \ar@<-1ex>_{pr\sb{0}}[d] &  M\sb{1} \ar@<-1ex>_{d\sb{0}}[d]\\
X\sb{0} \ar_{\Id}[r]\ar@<-1ex>_{\zD\sb{X\sb{0}}}[u] & X\sb{0} \ar@<-1ex>_{s\sb{0}}[u]
}
\end{equation*}
\nid in \w[,]{\Spl\clC}
where for each \w[,]{(a,b)\in X\sb{1}} with \w{d\sb{0}\bar{\zvt}(f)(a,b)=(\Id\circ p\sb{0})(a,b)=a} we have
\begin{equation*}
    \bar\zvt(f) \zD\sb{X\sb{0}}(a)=\bar\zvt(f)(a,a)=\zs(a,a) f(a)-f(a)\zs(a,a)\\
    = O_M=\zs(a,a)=s\sb{0}(a).
\end{equation*}
\nid Here \w{\zs \zD\sb{X\sb{0}}= s\sb{0}: X\sb{0}\rw M\sb{1}} because \w{(\zs,\Id):X\rw M} is a map of groupoids.
Thus we have a map in \w{\Hom\sb{\Sz{\Spl}\clC/RHY}(Y,\, R (M))}
given by the composite
\begin{equation*}
\xymatrix@C=3.5pc@R=1.6pc{
  X\sb{0} \ar^{t\sb{1}}[r] \ar@<-1ex>_{q}[d] & X\sb{1} \ar^{\bar\zvt(f)}[r] \ar@<-1ex>_{pr\sb{0}}[d] &
  M\sb{1} \ar@<-1ex>_{d\sb{0}}[d]\\
\pi\sb{0} \ar_{t}[r]\ar@<-1ex>_{t}[u] & X\sb{0} \ar_{\Id}[r] \ar@<-1ex>_{\zD\sb{X\sb{0}}}[u] & X\sb{0} \ar@<-1ex>_{s\sb{0}}[u]
}
\end{equation*}
\nid By \S \ref{sbs-bourne} this corresponds to the map
\w{\zvt'(f)$ in $\Hom\sb{\Sz\Gpd \clC/X}(X,M)} with \w[.]{\zvt'(f)\sb{0}=\Id} and \w[.]{\zvt'(f)\sb{1}=\bar\zvt(f)}
We define
\begin{myeq}\label{def-phi}
  \zf: \Hom\sb{\Sz\Gpd\, \clC/X} ((X\!\xrw{\Id}\! X),(M\!\xrw{\zr}\!X))\to
    \Hom\sb{\Sz\Gpd\, \clC/X} ((X\!\xrw{e}\!X),(e\sp{\ast} M\!\xrw{\zr'}\!X))
\end{myeq}
as follows. By \S \ref{sbs-bourne}, an element of
\w{\Hom\sb{\Sz\Gpd\, \clC/X} ((X\!\xrw{\Id}\! X),(M\!\xrw{\zr}\!X))} is determined by the map
\begin{myeq}\label{eq-maps-theta-xi-10}
\begin{gathered}
\xymatrix@C=4pc@R=1.7pc{
X\sb{1} \ar^{g}[r] \ar@<-1ex>_{pr\sb{0}}[d] &  M\sb{1} \ar@<-1ex>_{d\sb{0}}[d]\\
X\sb{0} \ar_{\Id}[r]\ar@<-1ex>_{\zD\sb{X\sb{0}}}[u] & X\sb{0} \ar@<-1ex>_{s\sb{0}}[u]
}
\end{gathered}
\end{myeq}
\nid in \w[.]{\Spl\clC} We associate to this another map in \w[:]{\Spl\clC}
\begin{myeq}\label{eq-maps-theta-xi-11}
\begin{gathered}
\xymatrix@C=4pc@R=1.7pc{
X\sb{1} \ar^{\widetilde\zf(g)}[r] \ar@<-1ex>_{pr\sb{0}}[d] & e\sp{\ast}M\sb{1} \ar@<-1ex>_{d'\sb{0}}[d]\\
X\sb{0} \ar_{e\sb{0}}[r]\ar@<-1ex>_{\zD\sb{X\sb{0}}}[u] & X\sb{0} \ar@<-1ex>_{s'\sb{0}}[u]
}
\end{gathered}
\end{myeq}
\nid where \w{\widetilde\zf(g)} into the pullback square of \wref{sbs-map-zvt-eq3} is
given by \w[,]{\psi_g\top e\sb{1}:X\sb{1}\to M\sb{1}\times X\sb{1}} with \w[.]{\psi_g(a,b)=\zs(tqa,a)g(a,b)\zs(b,tqb)}
Note that, since \w[:]{\zr\sb{1}(m)=(\pt\sb{0} m,\pt\sb{1} m)}
$$
\zr\sb{1} \psi\sb{g}(a,b)=\zr\sb{1}\zs\sb{1}(tqa,a)\zr\sb{1} g(a,b)\zr\sb{1}\zs(b,tqb)=(tqa,a)(a,b)(b,tqb)=e\sb{1}(a,b)~.
$$

We may check the commutativity of \wref[,]{eq-maps-theta-xi-11} using \wref[.]{estarmA}
The map \w{(e\sb{0}, \widetilde{\zf}(g))} in \wref{eq-maps-theta-xi-11} then defines an element \w{\zf(g)}
in \w[.]{\Hom\sb{\Sz\Gpd\, \clC/X} ((X\!\xrw{e}\!X),(e\sp{\ast} M\!\xrw{\zr'}\!X))}
\end{definition}

\begin{lemma}\label{lem-maps-theta-xi-3}
For \w{\zvt'} as in \wref{def-thetaprime} and $\zf$ as in \wref[,]{def-phi} the diagram
\mysdiag[\label{eq-maps-theta-xi-12}]{
    \Hom\sb{\clC/X\sb{0}}((X\!\xrw{\Id}\!X),(j\sp{\ast} M\!\xrw{\zl}\!X\sb{0}) \ar[rr]^(0.59){\zvt'} \ar^{(tq)\sp{\ast}}[d] &&
  \Hom\sb{\Sz\Gpd\,\clC/X}(X,M)\ar^{\zf}[d] \\
  \Hom\sb{\clC/X\sb{0}}((X\!\xrw{tq}\!X),(j\sp{\ast} M\!\xrw{\zl}\!X\sb{0}) \ar[rr]^(0.59){\zvt} &&
    \Hom\sb{\Sz\Gpd\,\clC/X}(X,e\sp{\ast} M)
}
\nid commutes, with vertical isomorphisms, where \w{(tq)\sp{\ast}(\Id,f)=(tq,\hat f)} for
\w{(\Id,f)\in\Hom\sb{\clC/X\sb{0}}(X\sb{0},j\sb{1}\sp{\ast} M\sb{1})} and \w[.]{\hat f(a)=\zs(tqa,a)f(a)\zs(a,tqa)}
\end{lemma}

\begin{proof}
For each \w{(a,b)\in X\sb{1}} we have \w[,]{\zt'(\Id,f)(a,b)=\zs(a,b)f(b)-f(a)\zs(a,b)} so
\begin{equation*}
\begin{split}
 \zf\zvt'(\Id,f)(a,b)=&\zs(tqa,a)\{\zs(a,b)f(b)-f(a)\zs(a,b)\}\zs(tqb,b)= \\
  =\,  & \zs(tqa,a)\zs(a,b)f(b)\zs(b,tqb)-\zs(tqa,a)f(a)\zs(a,b)\zs(b,tqb)= \\
  =\,  & \hat{f}(b)-\hat{f}(a)=\zvt(\hat f)(a,b)=\zvt(tq)\sp{\ast}(\Id,f)(a,b)
\end{split}
\end{equation*}
Thus \wref{eq-maps-theta-xi-12} commutes. The map \w{(tq)^{\ast}} is an isomorphism, with inverse
sending \w{(tq,g):X\sb{0}\rw e\sp{\ast} j\sb{1}\sp{\ast} M\sb{1}} to \w[,]{(\Id,\widetilde{g} ): X\sb{0}\rw j\sb{1}\sp{\ast} M\sb{1}}
where \w[.]{\widetilde{g}(a)=\zs(a,tqa)g(a)\zs(tqa,a)}
The map $\zf$ is an isomorphism by construction.
\end{proof}

%
%
\sect{Comonad resolutions for $\SO$-cohomology}\label{sec-comonad-res}

We now define the comonad on track categories which is used to
construct functorial cofibrant replacements, yielding a formula
for computing the $\SO$-cohomology of a track category. This will
provide crucial ingredients (Theorem \ref{the-com-res-2} and
Corollary \ref{cor-com-res-1}) for our main result, Theorem
\ref{the-maps-theta-xi-1}.

\begin{mysubsection}{Track categories}\label{sbs-track-cat}
Track categories, the objects of \w{\Trt=\Gpd(\Cath)} of \S
\ref{snac}, are strict 2-categories $X$ with object set $\clO$
(and functors which are identity on objects) such that for each
\w{a,b \in\clO} the category \w{X(a,b)} is a groupoid. The double
nerve functor provides an embedding \w{\Nb{2}:\Trt \hookrightarrow
\funcat{2}{\Set}}

The lower right corner of \w{\Nb{2}X} (omitting degeneracies),
appears as follows, with the vertical direction groupoidal, and
the horizontal categorical:
\begin{equation*}
\entrymodifiers={+++[o]}
 \xymatrix@R=20pt@C=20pt{
   \quad \cdots \quad \ar@<2ex>[r] \ar[r] \ar@<1ex>[r] \ar@<-2ex>[d] \ar[d] \ar@<-1ex>[d]&
   \tens{X\sb{11}}{X\sb{10}} \ar@<1ex>[r] \ar@<0ex>[r]  \ar@<-2ex>[d] \ar[d] \ar@<-1ex>[d] \ar@<1ex>[l] \ar@<2ex>[l]   &
    \clO  \ar@<-2ex>[d] \ar[d] \ar@<-1ex>[d] \ar@<1ex>[l]\\
    \tens{X\sb{11}}{\clO}\ar@<2ex>[r] \ar[r] \ar@<1ex>[r] \ar@<-1ex>[d] \ar@<0.ex>[d] \ar@<-1ex>[u] \ar@<-2.ex>[u] &
    X\sb{11} \ar@<1ex>[r] \ar@<0ex>[r]  \ar@<-1ex>[d] \ar@<0.ex>[d] \ar@<1ex>[l] \ar@<2ex>[l] \ar@<-1ex>[u] \ar@<-2.ex>[u] &
    \clO  \ar@<0.ex>[d] \ar@<-1ex>[d] \ar@<1ex>[l] \ar@<-1ex>[u] \ar@<-2.ex>[u] \\
    \tens{X\sb{10}}{\clO} \ar@<2ex>[r] \ar[r] \ar@<1ex>[r] \ar@<-1ex>[u] &
     X\sb{10} \ar@<1ex>[r] \ar@<0ex>[r] \ar@<1ex>[l] \ar@<2ex>[l] \ar@<-1ex>[u] &
    \clO \ar@<1ex>[l] \ar@<-1ex>[u]\\
}
\end{equation*}
\nid There is a functor \w{\Pz:\Trt\to\Cath} given by dividing out
by the 2-cells: that is,\w{(\Pz X)\sb{0} = \mathcal{O}} and \w[,]{(\Pz
X)\sb{1} = q X\sb{1}} where $X\sb{1}$ is the groupoid of 1- and 2-cells in $X$
and \w{q:\Gpd \rw \Set} is the connected component functor.

\begin{definition}\label{def-track-cat-1}
We say \w{X\in\Trt} is \emph{homotopically discrete} if, for each
$a,b\in \clO$, the groupoid $X(a,b)$ is an equivalence relation,
that is, a groupoid with no non-trivial loops.
\end{definition}

\begin{remark}\label{rem-track-cat-0}
By taking nerves in the groupoid direction we define
  \begin{myeq}\label{defofI}
  I:\Trt=\Gpd(\Cath)\rw\funcat{}{\Cath}=\SO\mi\Cat
\end{myeq}
\nid If \w{F:X\rw Y} in \w{\Trt} is a $2$-equivalence (so for each
\w[,]{a,b\in\clO} \w{F(a,b)} is an equivalence of groupoids),
\w{IF} is a Dwyer-Kan equivalence of the corresponding
$\SO$-categories. In particular, if \w{X\in\Trt} is homotopically
discrete, and \w{d\,\Pz X} is a track category with only identity
$2$-cells), this holds for the obvious \w[.]{F:X\rw\dd\Pz X}
\end{remark}
\end{mysubsection}

\begin{mysubsection}{The comonad $\pmb{\ck}$}\label{sbs-monad-k}
Taking \w{\clC=\Cath} in \S\ref{sbs-adjpair} yields a pair of
adjoint functors
$$
L:\Cath \leftrightarrows \Gpd(\Cath)=\Trt : U~.
$$
\nid Let \w{\Gro} be the category of reflexive graphs with object
set $\clO$ and morphisms which are identity on objects (where a
reflexive graph is a diagram \w{\xymatrix{ X\sb{1} \ar@<2ex>^{d\sb{0}}[r]
\ar@<1ex>_{d\sb{1}}[r] & X\sb{0} \ar@<2ex>^{s\sb{0}}[l] }} with \w[).]{d\sb{0}
s\sb{0}=d\sb{1} s\sb{0}=\Id} There are adjoint functors \w[,]{F:\Gro
\rightleftarrows \Cath :V} where $V$ is the forgetful functor and
$F$ is the free category functor. By composition, we obtain a pair
of adjoint functors
\begin{myeq}\label{monad-k-eq3}
  LF:\Gro \rightleftarrows \Trt : VU\;,
\end{myeq}
and therefore a comonad \w[,]{(\ck, \zve, \zd)$\;, $\ck =
LFVU:\Trt \rw \Trt} where $\zve$ is the counit of the adjunction
\eqref{monad-k-eq3}, \w[,]{\zd=LF(\eta)VU} and $\eta$  the unit of the
adjunction \eqref{monad-k-eq3}. For each \w{X\in\Trt} we obtain a
simplicial object \w{\ck\sb{\bl}X\in\funcat{}{\Trt}} with
\w{\ck\sb{n} X= \ck\sp{n+1}X} and face and degeneracy maps given by
\begin{equation*}
\pt_i=\ck^i\zve\ck^{n-i}: \ck^{n+1}X\rw \ck^{n}X\hs\text{and}\hs
\zs_i=\ck^i\zd\ck^{n-i}: \ck^{n+1}X\rw \ck^{n+2}X~.
\end{equation*}
\nid The simplicial object \w{\ck\sb{\bl}X} is augmented over $X$
via \w[,]{\zve:\ck\sb{\bl}X\rw X} and \w{\ck\sb{\bl}X} is a simplicial
resolution of $X$ (see \cite{Weibel1994}).

\begin{remark}\label{rem-monad-k-1}
  The augmented simplicial object \w{VU(\ck\sb{\bl}X)\xrw{VU\zve}
  VUX}
  is aspherical (see for instance \cite[Proposition 8.6.10]{Weibel1994}).
\end{remark}

\begin{remark}\label{rem-mon}
Given \w{X\in\Cath} by \wref{adjpair.eq1A} we see that \w{LX} is a
homotopically discrete track category (Definition
\ref{def-track-cat-1}), with \w[.]{\Pz LX = X} There are two
canonical splittings \w[,]{\Pz LX =X\rw (LX\sb{0})=X_s\cop X_t} given
by the inclusion in the $s$ or in the $t$ copy of $X$. Since
\w{\ck Y=L(FVUY)} for each \w[,]{Y \in\Trt} the same holds for
\w[.]{\ck Y}

Furthermore, since \w{FVUY} is a free category, so is \w[.]{\Pz \ck
Y=FVUY} Since $F$ preserves coproducts (being a left adjoint),
\w{(\ck Y)\sb{0}=FVUY\cop FVUY = F(VUY \cop VUY)} and, using
\wref[,]{adjpair.eq1} \w[.]{(\ck Y)\sb{1}=F(VUY \cop VUY \cop VUY \cop VUY)}
Thus both \w{(\ck Y)\sb{0}} and \w{(\ck Y)\sb{1}} are free categories.
Similarly, \w{(\ck Y)_r} is a free category for each \w[.]{r >1}
\end{remark}
\end{mysubsection}

\begin{mysubsection}{The comonad resolution}\label{sbs-com-res}
Let \w{\iov:\funcat{}{\Trt}\rw \funcat{}{\SO\mi\Cat}} be the
functor obtained by applying the internal nerve functor $I$ of
\wref{defofI} levelwise in each simplicial dimension \wh so for
\w[,]{X\in\Trt} \w[.]{\iov\ck\sb{\bl} X\in \funcat{}{\SO\mi\Cat}}
Similarly, \w{\ion:\SO\mi\Cat \rw \funcat{2}{\Set}} is obtained by
applying the nerve functor in the category direction; applying
this levelwise to \w{ \iov\ck\sb{\bl} X} yields
\begin{myeq}\label{trisimpl}
  W=\ion\, \iov\ck\sb{\bl} X \in \funcat{3}{\Set}\;.
\end{myeq}
Below is a picture of the corner of $W$, in which the horizontal
simplicial direction is given by the comonad resolution, the
vertical is given by the nerve of the groupoid in each track
category, and the diagonal is given by the nerve of the category
in each track category.
\begin{myeq}\label{com-res-eq1}
\begin{gathered}
  \includegraphics[width=10cm]{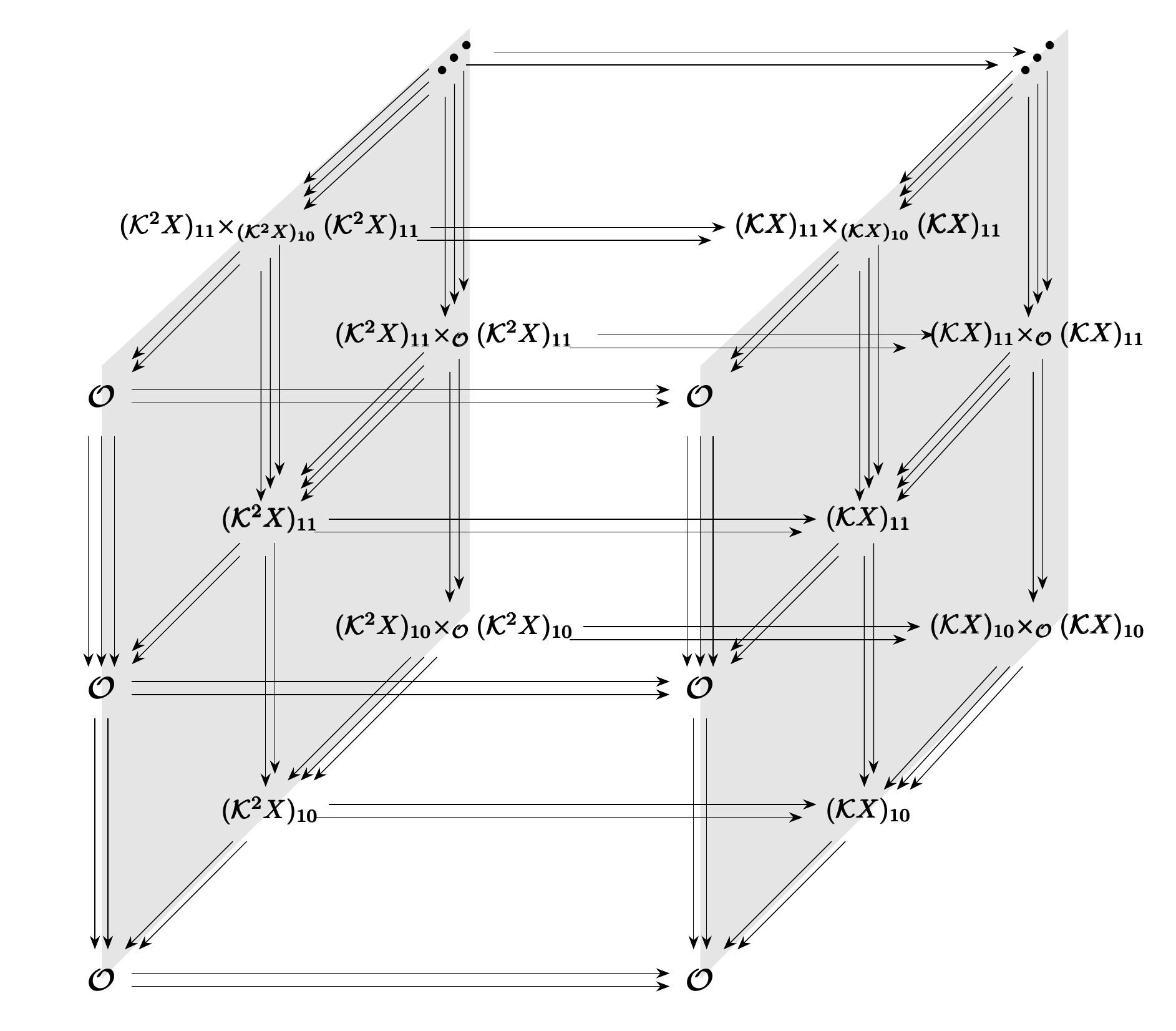}
\end{gathered}
\end{myeq}
\nid Note that the augmentation \w{\epsilon:\ck\sb{\bl} X\rw X} induces a map in \w[:]{\funcat{3}{\Set}}
\begin{myeq}\label{mapwtox}
   W \rw \ion\, \iov cX.
\end{myeq}
\nid Now let \w{Z=W^{(2)}} be $W$ thought of as a simplicial object in
\w{\funcat{2}{\Set}} along the direction appearing diagonal in the
picture, that is
\begin{myeq}\label{defofz}
  Z\in\funcat{}{\funcat{2}{\Set}}~,
\end{myeq}
\nid with \w{Z\sb{0}} the constant bisimplicial set at \w{\clO$, $Z\sb{1}}
given by
\begin{equation*}
\begin{split}
   & \xymatrix{&\hspace{47mm} &\vdots &} \\
   &
\mbox{\scriptsize
\xymatrix{
 \tens{(\ck^3 X)\sb{11}}{(\ck^3 X)\sb{10}} \ar@<1ex>[r] \ar@<0.0ex>[r]  \ar@<-1ex>[r] \ar@<1ex>[d] \ar@<0.0ex>[d]  \ar@<-1ex>[d]
& \tens{(\ck^2 X)\sb{11}}{(\ck^2 X)\sb{10}} \ar@<1ex>[d] \ar@<0.0ex>[d]  \ar@<-1ex>[d] \ar@<0.5ex>[r] \ar@<-0.5ex>[r]
& \tens{(\ck X)\sb{11}}{(\ck X)\sb{10}} \ar@<1ex>[d] \ar@<0.0ex>[d]  \ar@<-1ex>[d]\\
\!\!\!\!\!\!\!\cdots (\ck^3 X)\sb{11} \ar@<1ex>[r] \ar@<0.0ex>[r]  \ar@<-1ex>[r] \ar@<0.5ex>[d] \ar@<-0.5ex>[d]
&  (\ck^2 X)\sb{11} \ar@<0.5ex>[r] \ar@<-0.5ex>[r] \ar@<0.5ex>[d] \ar@<-0.5ex>[d]
& \ar@<0.5ex>[d] \ar@<-0.5ex>[d] (\ck X)\sb{11}\\
 (\ck^3 X)\sb{10} \ar@<1ex>[r] \ar@<0.0ex>[r]  \ar@<-1ex>[r]
&  (\ck^2 X)\sb{10} \ar@<0.5ex>[r] \ar@<-0.5ex>[r]
& (\ck X)\sb{10}
}
}
\end{split}
\end{equation*}
\nid with
\begin{myeq}\label{com-res-eq2}
  Z\sb{k}\cong \pro{Z\sb{1}}{Z\sb{0}}{k}\;,
\end{myeq}
\nid for each \w[.]{k\geq 2}

By applying the diagonal functor \w{\Diag: \funcat{2}{\Set}\rw
\funcat{}{\Set}} dimensionwise to $Z$ (viewed as in
\wref[),]{defofz} we obtain \w[.]{\diov Z\in
\funcat{}{\funcat{}{\Set}}}

To show that \w{\diov Z} is an \ww{\SO}-category, we must show that it behaves
like the nerve of a category object in simplicial sets in the outward
simplicial direction: this means that \w{(\diov Z)\sb{0}} is the ``simplicial set of
objects'' \wh and indeed it is the constant simplicial set at $\clO$.
Similarly, \w{(\diov Z)\sb{1}} is the ``simplicial set of arrows''.

Since \w{\Diag} preserves limits, by \wref{com-res-eq2} we have
\begin{equation*}
\begin{split}
(\diov Z)_k\cong& \Diag Z_k \cong \pro{\Diag Z\sb{1}}{\Sz\Diag Z\sb{0}}{k}\\
   = & \pro{(\diov Z)\sb{1}}{(\Sz\diov Z)\sb{0}}{k}\;.
\end{split}
\end{equation*}
\nid for each \w[.]{k\geq 2} Thus \w{\diov Z} is $2$-coskeletal, with unique fill-ins for
inner $2$-horns (the composite) so it is indeed in \w[,]{\SO\mi\Cat} and the map
\wref{mapwtox} induces a map \w{\za:\diov Z \rw I X} for $I$ as in \wref[.]{defofI}
\end{mysubsection}

\begin{lemma}\label{lem-com-res-1}
For $Z$ and \w{\Diag Z} as above, \w{\za:\diov Z \rw I X}
 is a Dwyer-Kan equivalence in \w[.]{\SO\mi\Cat}
\end{lemma}

\begin{proof}
We need to show that, for each \w{a,b\in\clO}
\begin{myeq}\label{com-res-eq2a}
  (\diov Z)(a,b)\rw X(a,b)
\end{myeq}
\nid is a weak homotopy equivalence.
Note that \w{(\diov Z)(a,b)} is the diagonal of
\begin{myeq}\label{com-res-eq3}
\begin{gathered}
\xymatrix@R=15pt@C=15pt{
&\ar@{}|(0.45){\vdots}[d] & \ar@{}|(0.45){\vdots}[d]  & \ar@{}|(0.45){\vdots}[d]  \\
\ar@{}|(0.3){\cdots}[r] & (\ck^3 X)\sb{11}(a,b) \ar@<1ex>[r] \ar@<0.0ex>[r]  \ar@<-1ex>[r] \ar@<-0.5ex>[d]   \ar@<0.5ex>[d]
& (\ck^2 X)\sb{11}(a,b) \ar@<-0.5ex>[r]   \ar@<0.5ex>[r] \ar@<-0.5ex>[d]   \ar@<0.5ex>[d]
& (\ck X)\sb{11}(a,b) \ar@<-0.5ex>[d]   \ar@<0.5ex>[d]\\
\ar@{}|(0.3){\cdots}[r] & (\ck^3 X)\sb{10}(a,b) \ar@<1ex>[r] \ar@<0.0ex>[r]  \ar@<-1ex>[r]
& (\ck^2 X)\sb{10}(a,b) \ar@<-0.5ex>[r]   \ar@<0.5ex>[r]
& (\ck X)\sb{10}(a,b),
}
\end{gathered}
\end{myeq}
and we have a map from \eqref{com-res-eq3} to the horizontally constant bisimplicial set
\begin{myeq}\label{com-res-eq4}
\begin{gathered}
\xymatrix@R=15pt@C=15pt{
& \ar@{}|(0.45){\vdots}[d] & \ar@{}|(0.45){\vdots}[d]  & \ar@{}|(0.45){\vdots}[d]  \\
\ar@{}|(0.3){\cdots}[r]& X\sb{11}(a,b) \ar@<1ex>[r] \ar@<0.0ex>[r]  \ar@<-1ex>[r] \ar@<-0.5ex>[d]   \ar@<0.5ex>[d]
& X\sb{11}(a,b) \ar@<-0.5ex>[r]   \ar@<0.5ex>[r] \ar@<-0.5ex>[d]   \ar@<0.5ex>[d]
& X\sb{11}(a,b) \ar@<-0.5ex>[d]   \ar@<0.5ex>[d]\\
\ar@{}|(0.3){\cdots}[r] & X\sb{10}(a,b) \ar@<1ex>[r] \ar@<0.0ex>[r]  \ar@<-1ex>[r]
& X\sb{10}(a,b) \ar@<-0.5ex>[r]   \ar@<0.5ex>[r]
& X\sb{10}(a,b) \\
}
\end{gathered}
\end{myeq}
inducing the map \eqref{com-res-eq2a} on diagonals.

We shall show that this map of bisimplicial sets from \wref{com-res-eq3} to \wref{com-res-eq4} is a
weak equivalence of simplicial sets in each vertical dimension. The corresponding map of diagonals
\wref{com-res-eq2a} is then a weak equivalence by \cite[Prop.1.7]{GJardS}.

Consider first vertical dimension $1$. We must show that the map of simplicial sets
\begin{myeq}\label{com-res-eq5}
  (\ck\sb{\bl} X)\sb{11}(a,b) \rw c X\sb{11}(a,b)
\end{myeq}
\nid is a weak equivalence, where \w{c X\sb{11}(a,b)} denotes the constant simplicial set at \w[.]{X\sb{11}(a,b)}
Note that \w{W\sb{11}=(V U W)\sb{1}} for each \w[,]{W\in\Trt} where $U$ is the internal arrow functor and
$V$ is the underlying graph functor (with \w{(VY)\sb{1}=Y\sb{1}} for each \w[),]{Y\in\Cath} so
\w{W\sb{11}(a,b)=(V U W)\sb{1}(a,b)} and thus \w{ (\ck\sb{\bl} X)\sb{11}(a,b)=(V U \ck\sb{\bl} X)\sb{1}(a,b)}
for each \w[.]{a,b\in\clO}

By Remark \ref{rem-monad-k-1}, the simplicial object \w{VU\ck\sb{\bl}X} is aspherical, so
\w{(V U \ck\sb{\bl}X)\sb{1}(a,b)} is, too, and is thus weakly equivalent to \w[.]{c (VUX)\sb{1}(a,b)=cX\sb{11}(a,b)}
Thus \wref{com-res-eq5} is a weak equivalence.

In vertical dimension $0$, from the vertical simplicial structure of \wref{com-res-eq3} we see
\begin{myeq}\label{com-res-eq6}
  (\ck\sb{\bl} X)\sb{10}(a,b) \rw c X\sb{10}(a,b)
\end{myeq}
\nid is a retract of \wref[,]{com-res-eq5}  so it is also a weak equivalence.

In vertical dimension $2$, we must show that
\begin{myeq}\label{com-res-eq7}
    \tens{(\ck\sb{\bl} X)\sb{11}(a,b)}{(\ck\sb{\bl} X)\sb{10}(a,b)} \rw \tens{c X\sb{11}(a,b)}{c X\sb{10}(a,b)}
\end{myeq}
is a weak equivalence. This is the induced map of pullbacks of the diagram
\begin{myeq}\label{com-res-eq8}
\begin{gathered}
\xymatrix{
(\ck\sb{\bl} X)\sb{11}(a,b) \ar^{\pt^\bl\sb{0}}[r] \ar[d] & (\ck\sb{\bl} X)\sb{10}(a,b)\ar[d] & (\ck\sb{\bl} X)\sb{11}(a,b) \ar_{\pt^\bl\sb{1}}[l] \ar[d]\\
c X\sb{11}(a,b) \ar^{\pt\sb{0}}[r] & c X\sb{10}(a,b) & c X\sb{11}(a,b) \ar_{\pt\sb{1}}[l]
}
\end{gathered}
\end{myeq}
\nid in $\clS$. By the above discussion, the vertical maps in \wref{com-res-eq8} are weak equivalences.

By definition of $\ck$ there is a pullback
\begin{myeq}\label{com-res-eq9}
\begin{gathered}
\xymatrix@R=16pt@C=25pt{
(\ck\sb{\bl}X)\sb{11}(a,b) \ar^{\pt^{\bl}\sb{1}}[r] \ar_{\pt^{\bl}\sb{0}}[d] & (\ck\sb{\bl}X)\sb{10}(a,b) \ar^{\nab}[d]\\
(\ck\sb{\bl}X)\sb{10}(a,b) \ar_{\nab}[r]  & \Pz (\ck\sb{\bl}X)\sb{10}(a,b)
}
\end{gathered}
\end{myeq}
in $\clS$, where
\begin{equation*}
  \nab:(\ck\sb{\bl}X)\sb{10}(a,b) = \Pz (\ck\sb{\bl}X)\sb{10}(a,b) \cop \Pz (\ck\sb{\bl}X)\sb{10}(a,b)\rw
  \Pz (\ck\sb{\bl}X)\sb{10}(a,b)\;.
\end{equation*}
\nid is the fold map.  To see that $\nab$ is a fibration, let \w[.]{Y\sb{\bl}=\Pz (\ck\sb{\bl}X)\sb{10}(a,b)}
For any commuting diagram
\begin{equation*}
\xymatrix@R=15pt@C=35pt{
\zL^k[n]\ar^{\za}[r] \ar@{^{(}->}_{j}[d] & Y\sb{\bl}\cop Y\sb{\bl} \ar^{\nab}[d]\\
\zD[n] \ar_{\zb}[r] & Y\sb{\bl}
}
\end{equation*}
\nid in $\clS$, $\za$ factors through
\w{i_t:Y\sb{\bl}\hookrightarrow Y\sb{\bl}\cop Y\sb{\bl}} \wb[,]{t=1,2}
since \w{\zL^k[n]} is connected, so
\begin{equation*}
\xymatrix@R=15pt@C=25pt{
\zL^k[n] \ar^{\za}[rr] \ar_{\za'}[rd] & &  Y\sb{\bl}\cop Y\sb{\bl} \\
 & Y\sb{\bl}\ar_{i_t}[ur] &
}
\end{equation*}
\nid commutes, and thus \w{\nab i_t \zb = \zb} and
\w[.]{i_t \zb j=i_t \nab \za = i_t \nab i_t \za' = i_t \za' = \za}

The maps \w{\pt\sb{0}^{\bl}} and  \w{\pt\sb{1}^{\bl}} are fibrations, since
they are pullbacks of such by \wref[.]{com-res-eq9}
The bottom horizontal maps in \wref{com-res-eq8} are fibrations since their target is discrete.
We conclude that the induced map of pullbacks \wref{com-res-eq7} is a weak equivalence.

Vertical dimension \w{i>2} is completely analogous.
\end{proof}

\begin{lemma}\label{lem-com-res-2}
  For any \w[,]{X\in\Trt} \w{\diov \ion\,\iov\ck\sb{\bl}X} is a free \ww{\SO}-category
  (see \S \ref{sbs-theta-cohom}).
\end{lemma}

\begin{proof}
  By Remark \ref{rem-mon}, for each track category \w{\ck\sb{r} X} the nerve in the
  groupoid direction is a free category in each simplicial degree. Thus for
  \w[,]{Z:=  \ion\,\iov\ck\sb{\bl}X}
\w{\diov Z} is also a free category in each simplicial degree. By \S \ref{sbs-monad-k},
the degeneracies \w{\zs_i :\ck^{n+1}X \rw \ck^{n+2}X} are given by \w[,]{\zs_i =\ck\sp{i}LF(\eta)VU\ck\sp{n-i}} where
  \w{\eta: \Id \rw V U L F} is the unit of the adjunction. Therefore, \w{\zs_i} sends generators to generators
  and so the same holds for the degeneracies of \w[,]{\diov Z} so it is a free $\SO$-category.
\end{proof}

\begin{corollary}\label{pro-comres-cohom-1}
  For any \w[,]{X\in\Trt} \w{\diov \ion\,\iov\ck\sb{\bl}X} is a cofibrant replacement of \w{I X} in \w[.]{\SO\mi\Cat}
\end{corollary}

\begin{proof}
By Lemma \ref{lem-com-res-1}, the map \w{\diov Z\rw I X} is a weak equivalence.
\end{proof}

We now show how to use the comonad resolution of a track category to compute
its $\SO$-cohomology:

\begin{proposition}\label{the-com-res-1}
For \w[,]{X\in\Trt} let \w{Z=\ion\,\iov\ck\sb{\bl}X} and let $M$ be a Dwyer-Kan
module over $X$. Then
\begin{myeq}\label{eq-com-res-1}
  \HSO{n-i}{I X}{M}~=~\pi\sb{i} \map\sb{\SO\mi\Sz{\Cat}/IX} (\diov Z,\ck\sb{X}(M,n))\;.
\end{myeq}
\end{proposition}

\begin{proof}
Let \w{\zf:\Diag\Fss X\rw IX} be the Dwyer-Kan standard free resolution of
\S \ref{sbs-theta-cohom}. Then, by definition,
\begin{myeq}\label{eq-com-res-2}
\HSO{n-i}{I X}{M}~=~\pi\sb{i} \map\sb{\SO\mi\Sz{\Cat}/IX}(\Diag\Fss X,\,\ck\sb{X}(M,n))\;,
\end{myeq}
\nid where \w{\ck\sb{X}(M,n)} is the twisted Eilenberg-Mac~Lane \ww{\SO}-category
of \S \ref{dsocoh}.

By Corollary \ref{pro-comres-cohom-1}, \w{\alpha:\diov Z\rw IX} is a cofibrant
replacement for \w[,]{IX} so given a commuting diagram
\begin{equation*}
\xymatrix@R=15pt@C=35pt{
{*} \ar[r] \ar@{^{(}->}[d] & \Diag\Fss X \ar^{\zf}[d] \\
\diov Z \ar_{\za}[r] & IX,
}
\end{equation*}
\nid there is a lift \w{\psi:\diov Z\rw\Diag\Fss X} with
\w{\zf\psi=\za} and $\psi$ a weak equivalence. Hence
\begin{equation*}
  \pi\sb{i}\map\sb{\SO\mi\Sz{\Cat}/IX}(\diov Z,\ck\sb{X}(M,n))\cong
  \pi\sb{i}\map\sb{\SO\mi\Sz{\Cat}/IX}(\Diag\Fss X,\ck\sb{X}(M,n))\
\end{equation*}
Thus \wref{eq-com-res-1} follows from \wref[.]{eq-com-res-2}
\end{proof}

\begin{theorem}\label{the-com-res-2}
  Let \w[,]{X\in\Trt} and $M$ a module over $X$; then for each \w[,]{s\geq 0}
  \w{\HSO{s}{I X}{M}=\pi\sp{s}\Cu} for the cosimplicial abelian group
\begin{myeq}\label{eqcosimpabgp}
  \Cu~:=~\pi\sb{1} \map\sb{\SO\mi\Sz{\Cat}/X}(\iov\ck\sb{\bl}X,M)~.
  \end{myeq}
\end{theorem}

\begin{proof}
  Since \w[,]{\diov Z=\Diag \iov \ck\sb{\bl} X} by Proposition \ref{the-com-res-1} and
  Lemma \ref{lem-sbs-bous-kan} we have
\begin{myeq}\label{eq-com-res-3}
\begin{split}
 \HSO{n-i}{IX}{M}~=&~\pi\sb{i}\map\sb{\SO\mi\Sz{\Cat}/IX}(\diov Z,\ck\sb{X}(M,n)\\
    =&~ \pi\sb{i}\Tot\,\map\sb{\funcat{2}{\Sz{\Cath}}/IX}(\iov\ck\sb{\bl} X,\ck\sb{X}(M,n))~.
\end{split}
\end{myeq}
\nid the homotopy spectral sequence of the cosimplicial space
\begin{equation*}
  W^{\bl}=\map\sb{\SO\mi\Sz{\Cat}}(\iov\ck\sb{\bl}X,\ck\sb{X}(M,n))
\end{equation*}
\nid (see \cite[X6]{BousKan}) has
\w{E\sp{s,t}\sb{2}=\pi\sp{s}\pi\sb{t} W^{\bl}\Rightarrow \pi\sb{t-s}\Tot\,W^{\bl}}
with
\begin{equation*}
   E^{s,t}\sb{1}~=~\pi\sb{t} \map(I\ck\sp{s} X,\ck\sb{X}(M,n))=\HSO{n-t}{I \ck\sp{s} X}{M}\;.
\end{equation*}
\nid Here we used the fact that \w{I\ck\sp{s} X} is a cofibrant $\SO$-category,
since it is free in each dimension and the degeneracy maps take generators to generators.

By Remark \ref{rem-mon}, \w{I\ck\sp{s} X} is homotopically discrete, so
\w{I \ck\sp{s} X \rw I\dd\Pz \ck\sp{s} X} is a weak equivalence.
Hence \w{I \ck\sp{s} X} is a cofibrant replacement of \w[,]{I\dd\Pz \ck\sp{s} X} and so
\begin{myeq}\label{eq-com-res-4}
  \HSO{n-t}{I \ck\sp{s} X}{M}=\HSO{n-t}{I\dd\Pz \ck\sp{s} X}{M} \;.
\end{myeq}
\nid Recall from \cite[Theorem 3.10]{BauesBlanc2011} that
\w{\HSO{s}{\dd \clC}{M}=H^{s+1}\sb{\BW}(\clC,M)}
for any category $\clC$ and \w[,]{s>0} where \w{H^{\bl}\sb{\BW}} is the Baues-Wirsching
cohomology of \cite{BauesWirsching1985}. Hence if $\clC$ is free,
\w{\HSO{s}{\dd \clC}{M}=0} for each \w[,]{s>0} by \cite[Theorem 6.3]{BauesWirsching1985}.

If \w[,]{n\neq t} it follows from
by \wref{eq-com-res-4} that \w[,]{E\sp{s,t}\sb{1}=\HSO{n-t}{I\ck\sp{s} X}{M}=0}
since \w{\Pz I\ck\sp{s} X} is a free category, so the spectral sequence collapses
at the \ww{E\sb{1}}-term, and
\begin{myeq}\label{eq-com-res-51}
  \HSO{s}{IX}{M}~=~\pi\sb{n-s}\Tot\,W^{\bl}~=~
  \pi\sp{s}\pi\sb{n} \map(\iov\ck\sb{\bl}X,\ck\sb{X}(M,n))\;.
\end{myeq}
\nid Since
\w[,]{\pi\sb{n}\map(\iov\ck\sb{\bl}X,\ck\sb{X}(M,n))=H\sp{0}\sb{\SO}(I\dd\Pz \ck\sb{\bl}X,M)}
this is independent of $n$. We deduce from \wref{eq-com-res-51} that
\w{\HSO{s}{IX}{M}=\pi\sp{s} \Cu} for \w{\Cu} as in \wref[.]{eqcosimpabgp}
\end{proof}

\begin{corollary}\label{cor-com-res-1}
For any \w[,]{X\in\Trt} $X$-module $M$, and \w{s\geq 0} we have
\begin{myeq}\label{eqcsag}
 \HSO{s}{I \dX\sb{0}}{j\sp{\ast} M}~\cong~\pi\sp{s}\hCu~,
\end{myeq}
\nid where \w{\hCu} is the cosimplicial abelian group
\w[.]{\pi\sb{1} \map\sb{\SO\mi\Sz{\Cat}/X}(\iov(\ck_\bl X)\sb{0},j\sp{\ast} M)}
\end{corollary}

\begin{proof}
  Let \w[.]{Z=N \iov\dd(\ck_\bl X)\sb{0}} Then \w{\diov Z \rw  IX\sb{0}}
  is a cofibrant replacement, by Corollary \ref{pro-comres-cohom-1}.
  Therefore, by Lemma \ref{lem-sbs-bous-kan}:
\begin{equation*}
\begin{split}
\HSO{n-i}{(I \dX\sb{0}}{j\sp{\ast} M}~=&~
\pi\sb{i}\map\sb{\SO\mi\Sz{\Cat}/I \dX\sb{0}}(\diov Z,\ck\sb{X\sb{0}}(j\sp{\ast} M,n))=\\
& =\pi\sb{i}\Tot\map\sb{\funcat{2}{\Sz{\Cath}}}(\iov d (\ck\sb{\bl}X)\sb{0}, j\sp{\ast} M)\;.
\end{split}
\end{equation*}
The homotopy spectral sequence for
\w{W\sp{\bl}=\map\sb{\SO\mi\Sz{\Cat}}(\iov\!\dd\,
  (\ck\sb{\bl}X)\sb{0},\ck\sb{X\sb{0}}(j\sp{\ast} M,n))}
again collapses at the \ww{E\sb{1}}-term, since
\w{(\ck\sp{s} X)\sb{0}} is a free category, yielding \wref[.]{eqcsag}
\end{proof}

%
%
\sect{$\SO$-cohomology of track categories and comonad cohomology}\label{sec-coh-track}

We now use the comonad of \S \ref{sbs-monad-k} to define the comonad cohomology
of a track category, rewrite the short exact sequence of Proposition \ref{pro-short-exact-1}
for \w[,]{\clC=\Cath} in terms of  mapping spaces, and use it to prove our main result,
Theorem \ref{the-maps-theta-xi-1}.

\begin{mysubsection}{Mapping spaces}\label{sbs-map-spaces}
  Given maps \w{f:A\rw B} and \w{g:M\rw B} in a simplicial model category $\clC$, we let
  \w{\map\sb{\clC/B}(A, M)} be the homotopy pullback
\begin{equation*}
\xymatrix@R=15pt@C=35pt{
\map\sb{\clC/B}(A, M) \ar[r] \ar[d] & \map\sb{\clC}(A, M) \ar^{g\sb{\ast}}[d]\\
\{f\} \ar[r] & \map\sb{\clC}(A, B)
}
\end{equation*}
\nid If $A$ is cofibrant and $g$ is a fibration, so is \w[.]{g\sb{\ast}} Moreover,
if \w{h:B\rw D} is a trivial fibration, so is  \w[.]{h\sb{\ast}:\map(A,B) \rw \map(A,D)}

Thus we obtain a map \w{\bar{h}\sb{\ast}:\map\sb{\clC/B}(A,M) \rw \map\sb{\clC/D}(A,M)}
defined by
\begin{equation*}
\xymatrix@C=6pt@R=1.3pc{
  \map\sb{\clC/B}(A,M) \ar[rrr] \ar[d] \ar^(0.7){\bar{h}\sb{\ast}}[rrdd] &&&
  \map\sb{\clC}(A,M) \ar[d] \ar^{\Id}[rrdd] \\
\{f\} \ar'[r][rrr] \ar[rrdd] &&&  \map\sb{\clC}(A,B) \ar^(0.4){h\sb{\ast}}[rrdd] \\
&& \map\sb{\clC/D}(A,M) \ar[rrr] \ar[d] &&& \map\sb{\clC}(A,M) \ar[d]\\
&& \{hf\} \ar[rrr] &&& \map\sb{\clC}(A,D)
}
\end{equation*}
\nid Since \w{h\sb{\ast}} is a weak equivalence, so is \w{\bar{h}\sb{\ast}}.

When \w[,]{\clC=\Trt} we will use this construction several times for \w{X=H Y}
as in \wref{splitobj} and \w[,]{e\sp{\ast}M} \w[,]{j\sp{\ast} M} and
\w{t\sp{\ast} j\sp{\ast} M} be as in\wref[,]{estarm}
\wref[,]{sbs-map-zvt-eq1} and \wref[,]{tjstarm} respectively:

\begin{enumerate}
\renewcommand{\labelenumi}{(\roman{enumi})~}
\item The diagram
\begin{equation*}
\xymatrix@C3pc@R=1.3pc{
  \{\Id_X\} \ar[r] \ar[d] & \map(X,X) \ar^{q_*}[d] & \map(X,M) \ar_{\zr\sb{\ast}}[l]
  \ar^{\Id}[d] \\
  \{q\} \ar[r] & \map(X,d \pi\sb{0}) & \map(X,M) \ar^{q_* \zr_*}[l]
}
\end{equation*}
\nid induces a map \w[,]{\bar{q}_*:\map\sb{\Sz\Trt/X}(X,M) \rw \map\sb{\Sz\Trt/\dpz}(X,  M)}
where $M$ is a track category over \w{\dpz} via the map
\w[,]{q:X\rw \dpz} which is a weak equivalence (since $X$ is homotopically
discrete) and a fibration (since \w{\dpz} is discrete). Hence
\w{\bar{q}\sb{\ast}} is a weak equivalence.
\item The diagram
\begin{equation*}
\xymatrix@C3pc@R=1.3pc{
  \{\zD\sb{X\sb{0}}\} \ar[r] \ar[d] & \map(\dX\sb{0},X) \ar^{t\sp{\ast}}[d] &
  \map(\dX\sb{0},j\sp{\ast} M) \ar_{}[l] \ar^{t\sp{\ast}}[d] \\
  \{\zD\sb{X\sb{0}}t\} \ar[r] & \map(\dpz, X) &
  \map(\dpz, t\sp{\ast} j\sp{\ast} M) \ar[l]
}
\end{equation*}
\nid induces
\w[.]{\bar{t}\sp{\ast}:\map\sb{\Sz\Trt/\dX\sb{0}}(\dX\sb{0}, j\sp{\ast} M) \rw
  \map\sb{\Sz\Trt/\dpz}(\dpz,\, t\sp{\ast} j\sp{\ast} M)}

Since \w[,]{q\sp{\ast} t\sp{\ast}=e\sp{\ast}} we have a diagram
\begin{equation*}
\xymatrix@C3pc@R=1.3pc{
  \{\zD\sb{X\sb{0}}t\} \ar[r] \ar[d] & \map(d \pi\sb{0},X) \ar^{q\sp{\ast}}[d] &
  \map(d \pi\sb{0}, t\sp{\ast} j\sp{\ast} M)
  \ar^{(\zr k e)\sb{\ast}}[l] \ar^{q\sp{\ast}}[d] \\
  \{\zD\sb{X\sb{0}}t\} \ar[r] & \map(\dX\sb{0}, X) &
  \map(\dpz, e\sp{\ast} j\sp{\ast} M) \ar^{}[l]
}
\end{equation*}
\nid inducing
\w[.]{\bar{q}\sp{\ast}:\map\sb{\Sz{\Trt/\dpz}}(\dpz,\, t\sp{\ast} j\sp{\ast} M) \rw
  \map\sb{\Sz{\Trt/\dpz}}(\dpz, e\sp{\ast} j\sp{\ast} M)}
\item The diagram
\begin{equation*}
\xymatrix@C3pc@R=1.3pc{
  \{\Id_X\} \ar[r] \ar[d] & \map(X,X) \ar^{j\sp{\ast}}[d] &
  \map(X,M) \ar_{\zr\sb{\ast}}[l] \ar^{j\sp{\ast}}[d] \\
\{j\} \ar[r] & \map(\dX\sb{0},X) & \map(\dX\sb{0},j\sp{\ast} M) \ar_{\zl\sb{\ast}}[l]
}
\end{equation*}
\nid induces
\w[.]{\bj\sp{\ast}:\map\sb{\Sz\Trt/X}(X,M) \rw
  \map\sb{\Sz\Trt/\dX\sb{0}}(\dX\sb{0}, j\sp{\ast} M)}
\item The diagram
\begin{equation*}
\xymatrix@C3pc@R=1.3pc{
  \{q\} \ar[r] \ar[d] & \map(X,d\pi\sb{0}) \ar^{t\sp{\ast}}[d] &
  \map(X,M) \ar_{q\sb{\ast}\zr\sb{\ast}}[l] \ar^{t\sp{\ast} j\sp{\ast}}[d] \\
  \{qjt\}=\{\Id\sb{d\pi\sb{0}}\} \ar[r] & \map(d \pi\sb{0},d \pi\sb{0}) &
  \map(\dpz,t\sp{\ast} j\sp{\ast} M) \ar^{q\sb{\ast} \zr\sb{\ast}}[l]
}
\end{equation*}
\nid induces
\w[.]{\bar{t}\sp{\ast}:\map\sb{\Sz\Trt/d\pi\sb{0}}(X, M) \rw
  \map\sb{\Sz\Trt/d \pi\sb{0}}(d \pi\sb{0}, t\sp{\ast} j\sp{\ast} M)}
Since \w{t\sp{\ast}} is a trivial fibration, \w{\bar{t}\sp{\ast}} is a weak equivalence.
\item The  diagram
\begin{equation*}
\xymatrix@C3pc@R=1.3pc{
  \{\Id_X\} \ar[r] \ar[d] & \map(X,X) \ar^{e\sp{\ast}}[d] &
  \map(X,M) \ar_{\zr\sb{\ast}}[l] \ar^{e\sp{\ast}}[d] \\
\{e\} \ar[r] & \map(X,X) & \map(X,e\sp{\ast} M) \ar^{\zr'\sb{\ast}}[l]
}
\end{equation*}
\nid induces
\w[.]{\bar{e}\sp{\ast}:\map\sb{\Sz\Trt/X}(X,M) \rw \map\sb{\Sz\Trt/X}(X, e\sp{\ast} M)}
Moreover,  \w{e:X\rw X} is a weak equivalence (since $X$ is homotopically discrete), so
\w{\bar{e}\sp{\ast}} is, too.
\item Using the fact that \w{q\sb{\ast}e\sp{\ast}=q\sb{\ast}} (since \w[),]{qe=qtq=q} we have
\begin{equation*}
\xymatrix@C3pc@R=1.3pc{
  \{e\} \ar[r] \ar[d] & \map(X,X) \ar^{q\sb{\ast}}[d] &
  \map(X,e\sp{\ast} M) \ar_{\zr'\sb{\ast}}[l] \ar^{r\sb{\ast}}[d] \\
\{q\} \ar[r] & \map(X,\dpz) & \map(X,M) \ar[l]
}
\end{equation*}
\nid which induces
\w[,]{\bar{q}'\sb{\ast}:\map\sb{\Sz\Trt/X}(X,e\sp{\ast} M) \rw
  \map\sb{\Sz\Trt/\dpz}(X, M)} and
\begin{myeq}\label{eq-map-spaces-9}
\begin{gathered}
\xymatrix@C=8pt@R=15pt{
  \map\sb{\Sz\Trt/X}(X,M) \ar^{e\sp{\ast}}[rr] \ar_{\bar{q}\sb{\ast}}[rd] &&
  \map\sb{\Sz\Trt/X}(X,e\sp{\ast} M) \ar^{\bar{q'}\sb{\ast}}[ld]\\
& \map\sb{\Sz\Trt/d\pi\sb{0}}(X, M)&
}
\end{gathered}
\end{myeq}
\nid commutes, since \w[.]{\bar{q'}\sb{\ast} e\sp{\ast}=q\sb{\ast}}
\item Finally, the diagram
\begin{equation*}
\xymatrix@C3pc@R=1.3pc{
  \{e\} \ar[r] \ar[d] & \map(X,X) \ar^{j\sp{\ast}}[d] &
  \map(X,e\sp{\ast} M) \ar_{\zr'\sb{\ast}}[l] \ar^{j\sp{\ast}}[d] \\
  \{jtq\} \ar[r] & \map(\dX\sb{0},X) &
  \map(\dX\sb{0}, e\sp{\ast} j\sp{\ast} M) \ar_{\zr\sp{\ast}}[l]
}
\end{equation*}
\nid induces
\w[.]{(\bj')\sp{\ast}:\map\sb{\Sz\Trt/X}(X,e\sp{\ast} M) \rw
  \map\sb{\Sz\Trt/\dX\sb{0}}(\dX\sb{0}, e\sp{\ast} j\sp{\ast} M)}
\end{enumerate}
\end{mysubsection}

\begin{lemma}\label{lem-maps-theta-xi-1}
For any map \w{f:A\rw B} in \w[,]{\Cath} with $A$ free, there is an isomorphism
\begin{equation*}
    \pi\sb{1} \map\sb{\Sz\Trt/B}(A,M)=\HSO{0}{A}{M} \cong \Hom\sb{\Sz\Cath/B}(A,M\sb{1})\;.
\end{equation*}
\end{lemma}

\begin{proof}
  Since $A$ is free, \w{\co{A}\in s\OC=\SOC}  is cofibrant, and is its own fundamental
  track category.
A module $M$, as an abelian group object in \w[,]{\Trt/B} is an Eilenberg-Mac~Lane object
\w[,]{\EM{B}{M\sb{1}}{1}} which explains the first equality. By \cite[II, \S 2]{QuiH},
the $n$-simplices of the cosimplicial abelian group
 \w{\map\sb{\Sz\Trt/B}(A,M)} are maps  of categories over $B$ of the form
 \w{\sigma:A\otimes\Delta[n]\rw M\sb{n}} where  \w{A\otimes\Delta[n]} is the coproduct
 in \w{\Cath} of one copy of $A$ for each $n$-simplex of \w[.]{\Delta[n]}

 Since \w[,]{M\sb{0}=B} \w{\map\sb{\Sz\Trt/B}(A,M)\sb{0}} is the singleton \w[.]{\{f\}}
 The non-degenerate part of \w{\map\sb{\Sz\Trt/B}(A,M)\sb{1}} is
 \w[,]{\Hom\sb{\Sz\Cath}(A,M\sb{1})} so the $1$-cycles are
 \w[.]{ \Hom\sb{\Sz\Cath/B}(A,M\sb{1})}
Since $M$, and thus \w[,]{\map\sb{\Sz\Trt/B}(A,M)}
 is a $1$-Postnikov section, the $1$-cycles are equal to
 \w[.]{\pi\sb{1} \map\sb{\Sz\Trt/B}(A,M)}
\end{proof}

\begin{lemma}\label{lem-maps-theta-xi-2}
For \w{X\in\Trt}  of the form \w{HY}  and $M$ an $X$-module, the diagram
\myrdiag[\label{eqlemtx}]{
  \pi\sb{1}\map\sb{\Sz\Trt/X}(X,M) \ar_{\|\wr}[d] \ar[rr]^(0.4){j\sp{\ast}} &&
  \Hom\sb{\Sz\Cath/X\sb{0}}((X\sb{0}\xrw{\Id}X\sb{0}),\
  (j\sb{1}\sp{\ast} M\sb{1}\to X\sb{0}) \ar_{\|\wr}[d]\\
  \Hom\sb{\Sz\Cath/\pi\sb{0}}(\pi\sb{0},\,t\sp{\ast} j\sb{1}\sp{\ast} M\sb{1})
  \ar[rr]_(0.39){q\sp{\ast}} &&
  \Hom\sb{\Sz\Cath/X\sb{0}}((X\sb{0}\xrw{e\sb{0}}X\sb{0}),\
  (e\sp{\ast} j\sb{1}\sp{\ast} M\sb{1}\to X\sb{0}))
}
\nid commutes, and there is an isomorphism
\begin{equation*}
  \omega: \Hom\sb{\Sz\Cath/X\sb{0}}((X\sb{0}\xrw{e\sb{0}}X\sb{0}),\
  (e\sp{\ast} j\sb{1}\sp{\ast} M\sb{1}\xrw{\zl\sb{1}}X\sb{0}))~\to~
  \Hom\sb{\Sz\Cath/X\sb{0}}\,(X\sb{0}\xrw{e\sb{0}}X\sb{0}),\ j\sb{1}\sp{\ast}
  M\sb{1}\xrw{\zl\sb{1}}X\sb{0}))
\end{equation*}
\nid such that \w[,]{\omega q\sp{\ast}=\xi} for $\xi$ as in \wref[.]{def-of-xi}
\end{lemma}

\begin{proof}
We have a commuting diagram
\begin{myeq}\label{eq-maps-theta-xi-1}
\begin{gathered}
\xymatrix@C=15pt@R=16pt{
  \map\sb{\Sz\Trt/X}(X,M) \ar^{\bj\sp{\ast}}[rr] \ar^{\bar{e}\sp{\ast}}[d]
  \ar@<-10ex>@/_1pc/_{\bar{q}\sp{\ast}}[dd] &&
  \map\sb{\Sz\Trt/\dX\sb{0}}(\dX\sb{0}, j\sp{\ast} M) \ar^{\tilde{e}\sp{\ast}}[d] \\
  \map\sb{\Sz\Trt/X}(X,e\sp{\ast} M) \ar^{(\bj')\sp{\ast}}[rr] \ar^{\bar{q'}\sb{\ast}}[d] &&
  \map\sb{\Sz\Trt/\dX\sb{0}}(\dX\sb{0}, e\sp{\ast} j\sp{\ast} M)\\
\map\sb{\Sz\Trt/d \pi\sb{0}}(X, M) \ar_{\bar{t}\sp{\ast}}[d] \\
\map\sb{\Sz\Trt/d \pi\sb{0}}(d \pi\sb{0} , t\sp{\ast} j\sp{\ast} M) \ar_{q\sp{\ast}}[uurr]
}
\end{gathered}
\end{myeq}
\nid in which the vertical maps are weak equivalences by \S \ref{sbs-map-spaces}.
Applying \w{\pi\sb{1}} yields
\begin{myeq}\label{eq-maps-theta-xi-2}
\begin{gathered}
\xymatrix@C=2pt@R15pt{
  \pi\sb{1}\map\sb{\Sz\Trt/X}(X,M) \ar_{\|\wr}[d] \ar[rr]^(0.5){j\sp{\ast}} &&
  \pi\sb{1}\map\sb{\Sz\Trt/\dX\sb{0}}\ar@<-8ex>_{\|\wr}^{\chi}[d]+<0pt,10pt>
  (\dX\sb{0},j\sp{\ast} M)  \\
  \pi\sb{1}\map\sb{\Sz\Trt/d \pi\sb{0}}(d \pi\sb{0},t\sp{\ast} j\sp{\ast} M)
  \ar[rr]_(0.43){q\sb{\ast}} && \pi\sb{1}\map\sb{\Sz\Trt/\dX\sb{0}}\,
((\dX\sb{0}\xrw{\dd e\sb{0}}\dX\sb{0}),\  e\sp{\ast} j\sp{\ast} M)
}
\end{gathered}
\end{myeq}
\nid with vertical isomorphisms. By Lemma \ref{lem-maps-theta-xi-1},
\begin{equation*}
\begin{split}
  \pi\sb{1}\map\sb{\Sz\Trt/\dX\sb{0}}(\dX\sb{0}, j\sp{\ast} M)\cong&
  \Hom\sb{\Sz\Cath/X\sb{0}} ((\dX\sb{0}\xrw{\Id}\dX\sb{0}),\ (j\sb{1}\sp{\ast}
  M\sb{1}\xrw{\zl\sb{1}}\dX\sb{0})) \\
  \pi\sb{1}\map\sb{\Sz\Trt/d \pi\sb{0}}(d \pi\sb{0}, t\sp{\ast} j\sp{\ast} M)\cong &
  \Hom\sb{\Sz\Cath/\pi\sb{0}}(\pi\sb{0},t\sp{\ast} j\sb{1}\sp{\ast} M\sb{1})\;, \\
  \pi\sb{1}\map\sb{\Sz\Trt/\dX\sb{0}}(\dX\sb{0}, e\sp{\ast} j\sp{\ast} M)\cong &
  \Hom\sb{\Sz\Cath/X\sb{0}}(X\sb{0}, q\sp{\ast}t\sp{\ast} j\sb{1}\sp{\ast} M\sb{1})\;.
\end{split}
\end{equation*}
Hence from \wref{eq-maps-theta-xi-2} we obtain
\begin{equation*}
\xymatrix@C=15pt@R=15pt{
  \pi\sb{1}\map\sb{\Sz\Trt/X}(X,M) \ar_{\|\wr}[d] \ar[rr]^(0.4){j\sp{\ast}} &&
  \Hom\sb{\Sz\Cath/X\sb{0}}((X\sb{0}\xrw{\Id}X\sb{0}),\  (j\sb{1}\sp{\ast} M\sb{1}\to
  X\sb{0})) \ar_{\chi}[d] \\
  \Hom\sb{\Sz\Cath/\pi\sb{0}}(\pi\sb{0},t\sp{\ast} j\sb{1}\sp{\ast} M\sb{1})
  \ar[rr]_(0.4){q\sp{\ast}} &&
  \Hom\sb{\Sz\Cath/X\sb{0}}((X\sb{0}\xrw{\Id}X\sb{0}),\  (e\sp{\ast}j\sb{1}\sp{\ast}
  M\sb{1}\to X\sb{0}))
}
\end{equation*}
\nid The right vertical map $\chi$ in \wref{eq-maps-theta-xi-2} sends
\w{\bar f:X\sb{0} \to j\sb{1}\sp{\ast} M\sb{1}} to
\w[,]{\bar f e:X\sb{0} \to e\sp{\ast} j\sb{1}\sp{\ast} M\sb{1}} where $\bar f$ is given
by \w{\Id\top f} into the pullback
\wref{eqjone} defining \w[.]{j^{*}M\sb{1}} We now rewrite
\begin{equation*}
  \chi: \Hom\sb{\Sz\Cath/X\sb{0}} ((X\sb{0}\!\xrw{\Id}\!X\sb{0}),(j\sb{1}\sp{\ast}
  M\sb{1}\to X\sb{0}))\to
  \Hom\sb{\Sz\Cath/X\sb{0}} ((X\sb{0}\!\xrw{e\sb{0}}\!X\sb{0}),(e\sp{\ast} j\sp{\ast}
  M\xrw{\zl\sb{1}}X\sb{0}))
\end{equation*}
\nid in a different form, in order to show that it is an isomorphism.
Note that the map of groupoids \w{\zr=(\zr\sb{1},\Id):M\to X}
satisfies \w{\zr\sb{1}(m)=(\pt\sb{0}(m),\pt\sb{1}(m))} for all \w[,]{m\in M\sb{1}}
where \w[.]{M=\parbox{20mm}{\xymatrix@C=1.5pc@R=1.0pc
    {M\sb{1} \ar@<1ex>^{\pt\sb{0}}[r] \ar@<-1ex>_{\pt\sb{1}}[r] & X\sb{0}}}}
Thus the map \w{\bar f=\Id\top f} into
\wref{eqjone} has \w{\zr\sb{1} f=j=\zD\sb{X\sb{0}}} so
\w[,]{\zr f(a)=(a,a)=(\pt\sb{0} f(a),\pt\sb{1} f(a))} and thus
$f$ takes \w{X\sb{0}} to \w[.]{\uset{a\in X\sb{0}}{\cop}M(a,a)} Similarly,
a map \w{\bar{g}:X\sb{0}\to e\sp{\ast} j^{*}M\sb{1} }
into the pullback defining \w{e\sp{\ast} j^{*}M\sb{1} } is given by
\w{g\top e\sb{0}:X\sb{0}\to M\sb{1}\times X\sb{0}}
with \w[,]{\zD\sb{X\sb{0}}e\sb{0}e\sb{0}=\zD\sb{X\sb{0}}e\sb{0}=\zr\sb{1} g} so
\w[.]{g:X\sb{0} \rw \uset{a\in X\sb{0}}{\cop}M(tqa,tqa)}
For each $a\in X\sb{0}$ there is an isomorphism \w{\omega: M(a,a) \rw M(tqa,tqa)}
taking \w{m\in M(a,a)} to \w[,]{\zs(tqa,a)\circ m \circ \zs(a,tqa)} whose inverse
\w{\omega\sp{-1}: M(tqa,tqa) \rw M(a,a)} takes \w{m'} to
\w[.]{\zs(a,tqa)\circ m' \circ \zs(tqa,a)}

Given \w[,]{f:X\sb{0} \rw \uset{a}{\cop}M(a,a)} there is a commuting diagram
\begin{equation*}
\xymatrix@R=15pt@C=35pt{
X\sb{0} \ar^(0.4){f}[r] \ar_{e\sb{0}}[d] & \uset{a}{\cop}M(a,a) \ar^{\omega}[d] \\
X\sb{0} \ar_(0.4){f}[r] & \uset{a}{\cop}M(tqa,tqa)
}
\end{equation*}
so that \w[.]{\chi(\bar f)=(e\sb{0},f e\sb{0})=(e\sb{0},\omega f)}
Thus  $\chi$ is an isomorphism with inverse given by
\w[.]{\chi\sp{-1}(g)=(\Id,\omega\sp{-1}g)}

Finally, the isomorphism
\begin{equation*}
  \omega: \Hom\sb{\Sz\Cath/X\sb{0}} ((\Id\sb{X\sb{0}}),\,(q\sp{\ast} t\sp{\ast}
  j\sb{1}\sp{\ast} M\sb{1}\xrw{\zl''\sb{1}}X\sb{0}))\to
      \Hom\sb{\Sz\Cath/X\sb{0}} ((tq),\,(j\sb{1}\sp{\ast} M\sb{1}\xrw{\zl\sb{1}}X\sb{0}))
\end{equation*}
\nid sends $h$ to \w[,]{\omega(h)=l\sb{1} v\sb{1} h} where the maps \w{l\sb{1}} and
\w{v\sb{1}} are given by
\begin{myeq}\label{eq-maps-theta-xi-4}
\begin{gathered}
\xymatrix@R=15pt@C=35pt{
  q\sp{\ast} t\sp{\ast} j\sb{1}\sp{\ast} M\sb{1} \ar^{\zl''\sb{1}}[d] \ar^{v\sb{1}}[r] &
  t\sp{\ast} j\sb{1}\sp{\ast} M\sb{1}  \ar^{\zl'\sb{1}}[d] \ar^{l\sb{1}}[r] &
  j\sb{1}\sp{\ast} M\sb{1} \ar^{\zl\sb{1}}[d] \ar^{k\sb{1}}[r] & M\sb{1} \ar^{\zr\sb{1}}[d]\\
X\sb{0} \ar_{q}[r] & \pi\sb{0} \ar_{t}[r] & X\sb{0} \ar_{\zD\sb{X\sb{0}}}[r] & X\sb{1}
}
\end{gathered}
\end{myeq}
Define \w{f'} in \w{\Hom\sb{\Sz\Cath/X\sb{0}}((X\sb{0}\xrw{tq}X\sb{0}),\
  (j\sp{\ast} M\xrw{\zl\sb{1}}X\sb{0}))} by
\w{tq\top f} into the pullback \wref[.]{eqjone}
Thus \w[,]{\zr\sb{1} f=\zD\sb{X\sb{0}}tq} which also implies
\w[.]{\zr\sb{1} f=\zD\sb{X\sb{0}}(tq)(tq)}
It follows that there is a map \w{f''} making the following diagram commute
\begin{equation*}
\xymatrix@C=4pc@R=1.5pc{
X\sb{0} \ar@/^1pc/^{f}[rrd] \ar^{f''}[rd] \ar@/_1pc/_{tq}[rdd]& & \\
& q\sp{\ast} t\sp{\ast} j\sp{\ast}M\sb{1} \ar^{k\sb{1} l\sb{1} v\sb{1}}[r]
\ar_{\zl''\sb{1}}[d] &  M\sb{1} \ar^{\zr\sb{1}}[d] \\
&  X\sb{0} \ar_{\zD\sb{X\sb{0}}tq}[r] & X\sb{1}
}
\end{equation*}
\nid We claim that \w[.]{f'=l\sb{1} v\sb{1} f''=\omega(f'')} In fact,
\w[,]{k\sb{1} f'=f=k\sb{1} l\sb{1} v\sb{1} f''} while by \eqref{eq-maps-theta-xi-4} we have
\w[.]{\zl\sb{1} f'=tq=tq\,tq=tq\zl''\sb{1} f''=t\zl'\sb{1} v\sb{1} f''=
  \zl\sb{1} l\sb{1} v\sb{1} f''}
Together these imply that \w[,]{f'=l\sb{1} v\sb{1} f''} as claimed.
Thus $\omega$ is surjective.

Assume given \w{h,g\in \Hom\sb{\Sz\Cath/X\sb{0}}(X\sb{0},\,
  q\sp{\ast} t\sp{\ast} j\sb{1}\sp{\ast} M\sb{1})} with \w[,]{\omega(h)=\omega(g)}
(i.e., \w[).]{l\sb{1} v\sb{1} h=l\sb{1} v\sb{1} g} Then
\w{k\sb{1} l\sb{1} v\sb{1} h=k\sb{1} l\sb{1} v\sb{1} g} and
\w[,]{\zl''\sb{1} h=tq = \zl''\sb{1} g}
so \w[.]{h=g} This shows that $\omega$ is injective, and thus an isomorphism.

To see that \w[,]{\omega q\sp{\ast}=\xi} Let
\w{h\in\Hom\sb{\Sz\Cath/X\sb{0}}(\pi\sb{0}, t\sp{\ast} j\sb{1}\sp{\ast} M\sb{1})} be given by
\w{f\top\Id:\pi\sb{0}\to M\sb{1}\times\pi\sb{0}} in the pullback
\wref[.]{tjstarm} Then \w{k\sb{1} l\sb{1} h=f} and \w[,]{\zl'\sb{1} h=\Id} and
so \w[.]{\zr\sb{1} fq = \zr\sb{1} k\sb{1} l\sb{1} h q=\zD\sb{X\sb{0}} t \zl'\sb{1} h
q=\zD\sb{X\sb{0}}tq=\zD\sb{X\sb{0}}(tq)(tq)} Hence there is a map \w{h'}
making
\begin{equation*}
\xymatrix@C=3pc@R=1.5pc{
X\sb{0} \ar@/^1pc/^{fq}[rrd] \ar^{h'}[rd] \ar@/_1pc/_{tq}[rdd]& & \\
& q\sp{\ast} t\sp{\ast} j^{*}M\sb{1} \ar^{k\sb{1} l\sb{1} v\sb{1}}[r] \ar_{\zl''\sb{1}}[d] &
M\sb{1} \ar^{\zr\sb{1}}[d] \\
&  X\sb{0} \ar_{\zD\sb{X\sb{0}}tq}[r] & X\sb{1}
}
\end{equation*}
\nid commute, and \w[,]{h'=q\sp{\ast}h} so that
\w[.]{\omega(q\sp{\ast}h)=\omega(h')=l\sb{1} v\sb{1} h'}
Moreover, \w{k\sb{1} l\sb{1} v\sb{1} h'=fq=k\sb{1} l\sb{1} h q} and
$$
\zD\sb{X\sb{0}} t \zl'\sb{1} v\sb{1} h' = \zD\sb{X\sb{0}} \zl\sb{1} l\sb{1} v\sb{1} h' =
\zr\sb{1} k\sb{1} l\sb{1} v\sb{1} h' = \zr\sb{1} f q=
\zD\sb{X\sb{0}}tq\,tq=\zD\sb{X\sb{0}}tq
$$
\nid so \w[,]{\zl'\sb{1} v\sb{1} h' =q \zD\sb{X\sb{0}}t \zl'\sb{1} v\sb{1} h'=
  q\zD\sb{X\sb{0}}tq=q=\zl'\sb{1} h q}
since \w{q\zD\sb{X\sb{0}}t=\Id} and \w[.]{\zl'\sb{1} h=\Id} This implies that
\w[.]{v\sb{1} h' = h q}
We deduce that \w[,]{\omega q\sp{\ast}(h)=l\sb{1} h q=\xi(h)} so \w[.]{\omega q\sp{\ast}=\xi}
\end{proof}

\begin{proposition}\label{pro-maps-theta-xi-1}
For any \w{X=HY} as in \wref{splitobj} and $X$-module $M$, there is a short exact sequence
\begin{myeq}\label{eq-maps-theta-xi-13}
\begin{split}
  0\rw \pi\sb{1}\map\sb{\Sz\Trt/X}(X,M)&\xrw{j\sp{\ast}}
  \pi\sb{1}\map\sb{\Sz\Trt/\dX\sb{0}}(dX\sb{0},j\sp{\ast} M)\\
  & \xrw{\zvt'} \Hom\sb{\Sz\Trt/X}(X,M)\rw 0
\end{split}
\end{myeq}
\end{proposition}

\begin{proof}
  This follows by taking \w{\clC=\Cath} in Proposition \ref{pro-short-exact-1},
  together with \wref{eq-maps-theta-xi-12}
  and the top right corner of \wref{eqlemtx} identified with
  \w[.]{\pi\sb{1}\map\sb{\Sz\Trt/\dX\sb{0}}(dX\sb{0},j\sp{\ast} M)}
\end{proof}

Let \w{X\in\Trt} and \w{M\in[(\Trt,X\sb{0})/X]\sb{\ab}} and
$$
\bar I \ck\sb{\bl} X\in\funcat{}{\SO\mi\Cat/IX}
$$
be as in Section \ref{sbs-com-res}. Note that the augmentation
\w{\zve:\ck\sb{\bl} X\rw X} can be thought of as a map to the constant simplicial object,
so a compatible sequence of maps \w{\zve\sb{n}:\ck\sb{n} X\rw X} in $\Trt$
allowing us to pull back $M$ to \w[.]{\zve\sp{\ast}\sb{n} M}

\begin{definition}\label{def-maps-theta-xi-1}
  For each \w{X\in \Trt} and \w[,]{M\in [(\Trt,X\sb{0})/X]\sb{\ab}}
    the comonad cohomology of $X$ with coefficients in $M$ is defined by
  \begin{equation*}
    H\sp{s}\sb{\Com}(X,M)=\pi\sp{s} \Hom\sb{\Sz\Trt/X}(\ck\sb{\bl}X,M)\;.
  \end{equation*}
\end{definition}

\begin{theorem}\label{the-maps-theta-xi-1}
  Assume given  \w{X\in\Trt} and \w[.]{M\in [(\Trt,X\sb{0})/X]\sb{\ab}} Then there is a
  long exact sequence of abelian groups
\begin{equation*}
\begin{split}
  \cdots \rw \HSO{n}{I X}{M}\rw& \HSO{n}{I \dX\sb{0}}{j\sp{\ast} M}\rw
  H\sp{n}\sb{\Com}(X,M)\rw \\
    & \rw \HSO{n+1}{I X}{M}\rw\cdots
\end{split}
\end{equation*}
\end{theorem}

\begin{proof}
  By Remark \ref{rem-mon}. \w{\ck\sb{n}X} is homotopically discrete for each $n$, and we
  can choose a splitting \w{\Pz \ck\sb{n} X\xrw{t\sb{n}}\ck\sb{n} X}
  to \w[,]{\ck\sb{n} X\xrw{q\sb{n}}\Pz \ck\sb{n} X} because \w{\ck\sb{n} X=LA}
  with \w{A=FUV X\sb{n-1}} and \w{q\sb{n}:A\sb{s}\cop A\sb{t} \rw A}
  (see \S \ref{sbs-adjpair}).
  By Proposition \ref{pro-maps-theta-xi-1} there is a short exact sequence of
  cosimplicial abelian groups
\begin{equation*}
\begin{split}
  0\rw & \pi\sb{1}\map\sb{\Sz\Trt/X}(I\ck\sb{\bl}X,M)\xrw{j\sp{\ast}}\pi\sb{1}
  \map\sb{\Sz\Trt/dX\sb{0}}(I\!\dd\,(\ck\sb{\bl}X)\sb{0},j\sp{\ast} M)\rw \\
    & \rw \Hom\sb{\Sz\Trt/X}(\ck_\bl X,M)\rw 0
\end{split}
\end{equation*}
where we use the augmentation $\zve_\bl:\ck_\bl X\rw X$ to pull back $M$ to $\zve_\bl M$.

We therefore obtain a corresponding long exact sequence
\begin{equation*}
\begin{split}
  \rw \pi\sp{s}\pi\sb{1}\map\sb{\Sz\Trt/X}(\ck\sb{\bl}X,M)&\xrw{j\sb{\ast}}
  \pi\sp{s}\pi\sb{1} \map\sb{\Sz\Trt/dX\sb{0}}(\dd\,(\ck\sb{\bl}X)\sb{0},j\sp{\ast} M)\rw \\
    & \rw \pi\sp{s}\Hom\sb{\Sz\Trt/X}(\ck_\bl X,M)\rw \cdots\;.
\end{split}
\end{equation*}
By definition, \w[,]{\pi\sp{s} \Hom\sb{\Sz\Trt/X}(\ck_\bl X,M)=H\sp{s}_C(X,M)}
with
$$
\pi\sp{s}\pi\sb{1} \map\sb{\Sz\Trt/X}(\bar I\ck\sb{\bl}X,M)= \HSO{s}{I X}{M}
$$
\nid and \w{\pi\sp{s}\pi\sb{1} \map\sb{\Sz\Trt/\dX\sb{0}}(\dd\,(\ck_\bl X)\sb{0},j\sp{\ast} M)=\HSO{s}{I\!\dX\sb{0}}{j\sp{\ast} M}}
by Theorem \ref{the-com-res-2} and Corollary \ref{cor-com-res-1}.
\end{proof}

\begin{corollary}\label{cor-maps-theta-xi-1}
For \w{X\in\Trt} with \w{X\sb{0}} a free category and \w{M\in[(\Trt,X\sb{0})/X]\sb{\ab}}
\begin{myeq}\label{eq-maps-theta-xi-141}
  \HSO{n+1}{IX}{M}\cong H\sp{n}\sb{\Com}(X,M)
\end{myeq}
\nid for each \w[.]{n\geq 1}
\end{corollary}
\begin{proof}
  Recall from \cite[Theorem 3.10]{BauesBlanc2011} that
  \w[,]{ \HSO{n}{I\!\dX\sb{0}}{j\sp{\ast}M}\cong H\sb{\BW}^{n+1}(I X\sb{0},j\sp{\ast} M)}
  and, since $X\sb{0}$ is free,   \w{H\sb{\BW}^{n+1}(IX\sb{0},j\sp{\ast}M)=0} by \cite[Theorem 6.3]{BauesWirsching1985}.
Thus the long exact sequence of Theorem \ref{the-maps-theta-xi-1} yields
\wref{eq-maps-theta-xi-141} for each \w[.]{n\geq 1}
\end{proof}

\begin{lemma}\label{lem-maps-theta-xi-4}
There is a functor \w{S:\Trt \rw \Trt} such that \w{(s_X)\sb{0}} is a free category,
for each \w[,]{X\in\Trt} with a natural $2$-equivalence \w{s_X:S(X)\rw X}
\end{lemma}

\begin{proof}
 Given \w{X\in\Gpd\;\clC} and a map \w{f\sb{0}:Y\sb{0}\rw X\sb{0}} in $\clC$, consider the pullback
\begin{myeq}\label{eq-maps-theta-xi-14}
\begin{gathered}
\xymatrix@C=40pt@R=17pt{
Y\sb{1} \ar[r] \ar_{f\sb{1}}[d] & Y\sb{0}\times Y\sb{0} \ar^{f\sb{0}\times f\sb{0}}[d]\\
X\sb{1} \ar_{(\pt\sb{0},\pt\sb{1})}[r] &  X\sb{0}\times X\sb{0}
}
\end{gathered}
\end{myeq}
\nid in $\clC$. Then there is \w{X(f\sb{0})\in\Gpd\,\clC} with \w{(X(f\sb{0}))\sb{0}=Y\sb{0}} and \w[,]{X(f\sb{0}))\sb{1}=Y\sb{1}}
such that  \w{(f\sb{0},f\sb{1}):X(f\sb{0})\rw X} is an internal functor.

Now let \w{\zve\sb{X\sb{0}}:FV X\sb{0} \rw X\sb{0}} be the counit of the adjunction
\w{F\dashv V} of  \S\ref{sbs-monad-k}, and let \w[,]{SX:=X(\zve\sb{X\sb{0}})} where \w{(\zve\sb{X\sb{0}})\in\Trt} is
the construction \wref[.]{eq-maps-theta-xi-14} Then \w{(SX)\sb{0}=FVX\sb{0}} is a free category, and there
is a map  \w{s_X:SX\rw X} in \w[.]{\Trt} We wish to show that it is a $2$-equivalence.

Since \w{s_X} is the identity on objects, to  it suffices to show that for each \w[,]{a,b\in\clO}  the map
\w{s_X(a,b):S(X)(a,b)\rw X(a,b)} is an equivalence of categories. The pullback
\begin{myeq}\label{eq-maps-theta-xi-16}
\begin{gathered}
\xymatrix@C=20pt@R=18pt{
(S(X))\sb{1} \ar[rr] \ar_{{s_X}\sb{1}}[d] && FVX\sb{0}\times FVX\sb{0} \ar^{\zve\sb{X\sb{0}}\times \zve\sb{X\sb{0}}}[d]\\
X\sb{1} \ar[rr] && X\sb{0}\times X\sb{0}
}
\end{gathered}
\end{myeq}
in \w{\Cat} induces a pullback of sets
\begin{equation*}
\xymatrix@C=10pt@R=18pt{
  \{S(X)(a,b)\}\sb{1} \ar[rrr] \ar_{\{s_X(a,b)\}\sb{1}}[d] &&&
  \{FVX\sb{0}(a,b)\times FVX\sb{0}(a,b)\}\sb{1} \ar^{\zve\sb{X\sb{0}}\times \zve\sb{X\sb{0}}}[d]\\
\{X(a,b)\}\sb{1} \ar[rrr] &&& \{X\sb{0}(a,b)\times X\sb{0}(a,b)\}\sb{1}
}
\end{equation*}
Thus for each \w[,]{(c,d)\in \{FVX\sb{0}(a,b)\times FVX\sb{0}(a,b)\}\sb{1}} the map \w{\{s_X(a,b)\}(c,d)} is a bijection.
Thus  \w{s_X(a,b)} is fully faithful. Since \w{(FVX\sb{0})\sb{1}\rw X\sb{01}} is surjective, as is
\w[,]{(FVX\sb{0})(a,b)\rw X\sb{0}(a,b)} \w{s_X(a,b)} is surjective on objects, so it is an equivalence of categories.
\end{proof}

We finally use our previous results to conclude that the $\SO$-cohomology of a track category can always be calculated from a comonad cohomology.
\begin{corollary}\label{cor-maps-theta-xi-2}
  For \w[,]{X\in\Trt} and $M$ an $X$-module, \w{\HSO{n+1}{IX}{M}\cong H\sp{n}_C(S(X),M)}
for each \w[,]{n>1} where $S(X)$ is as in Lemma \ref{lem-maps-theta-xi-4}.
\end{corollary}

\begin{proof}
  By Lemma \ref{lem-maps-theta-xi-4} the map $s_X:S(X)\rw X$ is a 2-equivalence in \w[,]{\Trt}
  since \w{(s_X)\sb{0}=\zve\sb{X\sb{0}}} is bijective on objects. Hence \w{Is_X} is a Dwyer-Kan equivalence
  in \w[,]{\SO\mi\Cat} so  \w[.]{\HSO{n}{IX}{M}\cong\HSO{n}{IS(X)}{M}}
By Lemma \ref{lem-maps-theta-xi-4}, \w{SX} satisfies the hypotheses of Corollary
\ref{cor-maps-theta-xi-1}, so also \w[.]{\HSO{n+1}{IS(X)}{M}\cong H\sp{n}\sb{\Com}(S(X),M)}
\end{proof}

%
%
\sect{The groupoidal case}\label{sec-groupoidal}

A $2$-groupoid is a special case of  a track category in which all cells are (strictly)
invertible. The category of such is denoted by \w[,]{\GGpdo=\Gpd(\Gpdo)}
with \w{i:\GGpdo\hookrightarrow \Trt} the full and faithful inclusion.

For \w{X\in\GGpdo} and \w{W=\ion\,\iov\ck_\bl X\in\funcat{3}{\Set}} as in
\S \ref{sbs-com-res}, let \w{S=W\up{1}\in\funcat{}{\funcat{2}{\Set}}} be $W$ thought
of as a simplicial object along the horizontal direction.
Thus for each \w[,]{i\geq 0} \w{S_i} is the bisimplicial set
\begin{equation*}
\mbox{\scriptsize
\xymatrix@C=3pc@R=1.5pc{
\cdots \ar@<1ex>[r] \ar@<0.0ex>[r]  \ar@<-1ex>[r] \ar@<1ex>[d] \ar@<0.0ex>[d]  \ar@<-1ex>[d]
& \tens{(\ck^{i+1} X)\sb{11}}{(\ck^{i+1} X)\sb{10}} \ar@<1ex>[d] \ar@<0.0ex>[d]  \ar@<-1ex>[d]
\ar@<0.5ex>[r] \ar@<-0.5ex>[r] & \clO \ar@<1ex>[d] \ar@<0.0ex>[d]  \ar@<-1ex>[d]\\
\tens{(\ck^{i+1} X)\sb{11}}{\clO} \ar@<1ex>[r] \ar@<0.0ex>[r]  \ar@<-1ex>[r] \ar@<0.5ex>[d]
\ar@<-0.5ex>[d]
&  (\ck^{i+1} X)\sb{11} \ar@<0.5ex>[r] \ar@<-0.5ex>[r] \ar@<0.5ex>[d] \ar@<-0.5ex>[d]
& \clO \ar@<0.5ex>[d] \ar@<-0.5ex>[d]\\
\tens{(\ck^{i+1} X)\sb{10}}{\clO} \ar@<1ex>[r] \ar@<0.0ex>[r]  \ar@<-1ex>[r]
&  (\ck^{i+1} X)\sb{10} \ar@<0.5ex>[r] \ar@<-0.5ex>[r]
& \clO
}}
\end{equation*}

Applying \w{\Diag:\funcat{2}{\Set}\rw \funcat{}{\Set}=\clS} dimensionwise
to \w{W\up{1}} we obtain \w[,]{\diov S\in \funcat{}{\funcat{}{\Set}}} with
\begin{myeq}\label{eq-groupoidal-1}
\Diag\!\up{3}W=\Diag\diov S=\Diag\diov Z
\end{myeq}
\nid for \w{Z:=W\up{2}\in\funcat{}{\funcat{2}{\Set}}} as in \wref[.]{defofz}

\begin{definition}\label{def-groupoidal-1}
  The \emph{classifying space} of \w{X\in\GGpdo} is \w[,]{BX=\Diag\Nb{2}X}
  where \w{\Nb{2}:\GGpdo\rw \funcat{2}{\Set}}  is the double nerve functor.
\end{definition}

\begin{remark}\label{rem-groupoidal-1}
By Lemma \ref{lem-com-res-1}, \w{\diov Z\rw IX} is a Dwyer-Kan
equivalence, where we think of \w{\diov Z} as an \ww{\SO}-category
by the discussion preceding the Lemma \ref{lem-com-res-1}. Conversely, we may also
think of \w{IX} as a bisimplicial set (implicitly, by applying the
nerve functor in the category direction), with \w{(\diov
Z)\sb{0}=(IX)\sb{0}=c(\clO)} (cf.\ \S \ref{snac}). Moreover,
\w{(\diov Z)\sb{j}\rw(IX)\sb{j}} is a weak homotopy equivalence
for all \w[.]{j\geq 0} Hence \w{\diov Z\rw IX} induces a weak
homotopy equivalence of diagonals
  \w[.]{\Diag\diov Z\simeq \Diag IX}
Since \w[,]{\Diag IX=\Diag\Nb{2}X=BX} by \wref{eq-groupoidal-1} we have
\begin{myeq}\label{eq-groupoidal-1bis}
  BX\simeq\Diag\!\up{3}W\;.
\end{myeq}
\nid The cohomology groups of  \w{X\in\GGpdo} with coefficients in  an $X$-module $M$
are defined to be \w[.]{H\sp{n-t}(BX,M)=\pi\sb{t}\map\sb{\funcat{}{\Sz\Set}}(BX,\ck(M,n))}
 By \wref{eq-groupoidal-1bis} this can be written as
\w[.]{\pi\sb{t} \map\sb{\funcat{}{\Sz\Set}}(\Diag\!\up{3}W,\ck(M,n))}
\end{remark}

\begin{lemma}\label{lem-groupoidal-11}
  Given \w[,]{\clC\in\Cath} with \w{N\clC\in\clS} viewed as a discrete $\SO$-category, and $M$ be a
  $\clC$-module, we have \w{H\sp{n}(B\clC,M)\cong \HSO{n}{N\clC}{M}} for each \w[.]{n\geq 0}
\end{lemma}
\begin{proof}
Since \w{N\clC} is a discrete $\SO$-category, we see that
\w[,]{\map\sb{\Sz{\SO}\mi\Cat}(N\clC,\ck(M,n))} is \w[,]{\map\sb{\clS}(N\clC,\ck(M,n))}
so \w[.]{\HSO{n}{\clC}{M}=  \pi\sb{0} \map\sb{\clS}(N\clC,\ck(M,n))=H\sp{n}(B\clC,M)}
\end{proof}
\begin{proposition}\label{pro-groupoidal-11}
  For \w{X\in\GGpdo} and $M$ an $X$-module,
  \w{H\sp{s}(BX,M)=\HSO{s}{X}{M}} for any \w{s\geq 0}
\end{proposition}
\begin{proof}
By \wref{eq-groupoidal-1} and Lemma \ref{lem-sbs-bous-kan}
\begin{equation*}
\begin{split}
    \map\sb{\clS}(\Diag\!\up{3}W,\ck(M,n))&=\map\sb{\clS}(\Diag\diov S,\ck(M,n))\cong\\
   \cong &\Tot\map\sb{\clS}(\diov S,\ck(M,n))\;.
\end{split}
\end{equation*}
\nid Therefore, the homotopy spectral sequence for
\w[,]{W\sp{\bl}=\map\sb{\Sz{\funcat{}{\Set}}}(\diov S,\ck(M,n))}
with \w[,]{E\sp{s,t}\sb{2}=\pi\sp{s}\pi\sb{t} W\sp{\bl} \Rightarrow  \pi\sb{t-s}\Tot W^{\bl}}
has \w[.]{E\sp{s,t}\sb{1} = \pi\sb{t} \map((\diov S)\sb{s},\ck(M,n))}
But \w[,]{(\diov S)\sb{s}=\Diag I\ck\sp{s+1}X\simeq I\Pz \ck\sp{s+1}X} since
\w{\ck\sp{s+1}X} is homotopically discrete, so
\begin{equation*}
E\sp{s,t}\sb{1}= H\sp{n-t}(BI\Pz \ck\sp{s+1}X,M)=\HSO{n-t}{I\Pz \ck\sp{s+1}X}{M}\;,
\end{equation*}
by Lemma \ref{lem-groupoidal-11}.
Since \w{\ck\sp{s+1}X} is free, by \cite[Theorem 3.10]{BauesBlanc2011} we have
\begin{equation*}
E\sp{s,t}\sb{1}=\HSO{n-t}{I\Pz \ck\sp{s+1}X}{M}=H\sp{n-t+1}\sb{\BW}(I\Pz \ck\sp{s+1}X,M)=0
\end{equation*}
\nid for \w[.]{n\neq t} Thus the spectral sequence collapses at the $E\sb{1}$-term and
\begin{myeq}\label{eq-groupoidal-2}
  H\sp{s}(BX,M)=\pi\sb{n-s}\Tot W\sp{\bl}=\pi\sp{s}\pi\sb{n} W\sp{\bl}
   =\pi\sp{s}\pi\sb{n}\map\sb{\clS}(\diov S,\ck(M,n))\;.
\end{myeq}
Since \w{(\diov S)\sb{s}\rw I\Pz\ck\sp{s+1}X} is a weak equivalence for all $s\geq 0$,
\begin{equation*}
      \pi\sb{n}\map\sb{\Sz{\funcat{}{\Set}}}(\diov S,\ck(M,n))=
    H\sp{0}(B\,\diov S,M)\cong H\sp{0}(B\,I\, \Pz\ck\sb{\bl}X,M)
\end{equation*}
\nid Thus \w{H\sp{s}(BX,M)=\pi\sp{s} H\sp{0}(B\,I\, \Pz \ck\sb{\bl} X,M)}
by \wref[.]{eq-groupoidal-2}
On the other hand, in the proof of Theorem \ref{the-com-res-2} we showed that
\w[,]{\HSO{s}{X}{M}=\pi\sp{s} \HSO{0}{I\, \Pz\ck\sb{\bl} X}{M}}
while by Lemma \ref{lem-groupoidal-11} we have
\w[.]{H\sp{0}(B\,\Pz \ck\sb{\bl} X,M)=\HSO{0}{I\,\Pz\ck\sb{\bl}X}{M}}

It follows that \w[.]{H\sp{s}(BX,M)=\HSO{s}{X}{M}}
\end{proof}

From Theorem \ref{the-maps-theta-xi-1} and Proposition \ref{pro-groupoidal-11} we deduce:

\begin{corollary}\label{the-groupoidal-1}
Any \w{X\in\GGpdo} and $X$-module $M$ have a long exact sequence
\begin{equation*}
\rw H\sp{n}(BX,M)\rw H\sp{n}(BX\sb{0},j\sp{\ast} M)\rw H\sp{n}_C(X,M)\rw H^{n+1}(BX,M)\rw\cdots\;.
\end{equation*}
\end{corollary}

\begin{remark}
A $2$-groupoid with a single  object is an internal groupoid in the category of groups,
equivalent to a crossed module. It can be shown that in this case the long exact sequence of
Corollary \ref{the-groupoidal-1} recovers \cite[Theorem 13]{PaoliHHA2003}.
\end{remark}

\end{document}